\documentclass[twoside]{article}

\usepackage[preprint]{aistats2026}
\usepackage[utf8]{inputenc}
\usepackage[T1]{fontenc}
\usepackage{mathtools}
\usepackage{amsmath,amssymb,amsthm,mathtools}
\usepackage{mathrsfs,bm,color}
\usepackage{fancyhdr,lastpage}
\usepackage[shortlabels]{enumitem}
\usepackage{calc}
\usepackage{bbm}
\usepackage{subcaption}
\usepackage{graphicx}
\usepackage{upgreek}
\usepackage{lmodern}
\usepackage[normalem]{ulem}
\usepackage{verbatim}
\usepackage{bbm}
\usepackage{stmaryrd}
\usepackage{amsmath}
\usepackage{amssymb}
\usepackage{dsfont}
\usepackage{amsthm}
\usepackage{hyperref}
\hypersetup{
	colorlinks,
	linkcolor={red!50!black},
	citecolor={red!80!black},
	urlcolor={blue!80!black}
}

\usepackage{xcolor}

\usepackage{graphicx}

\usepackage{amsfonts}%
\usepackage{dsfont}

\allowdisplaybreaks

\numberwithin{equation}{section}

\usepackage{thmtools}
\declaretheorem[name=Theorem,numberwithin=section,thmbox=S]{theo}
\declaretheorem[name=Proposition,numberwithin=section,thmbox=S]{prop}

\declaretheorem[name=Lemma, numberwithin=section]{lem}
\declaretheorem[name=Corollary,numberwithin=section]{cor}

\theoremstyle{definition}

\theoremstyle{plain}

\makeatletter
\renewenvironment{proof}[1][\proofname]{\par
  \pushQED{\qed}%
  \normalfont
  \topsep0pt \partopsep0pt \itemsep0pt \parsep0pt 
  \trivlist
  \item[\hskip\labelsep
    \scshape 
    #1\@addpunct{.}]\ignorespaces
}{%
  \popQED\endtrivlist\@endpefalse
}
\makeatother

\renewcommand{\and}{\mbox{ and }}

\newcommand{\integset}[1]{\ensuremath{ \llbracket #1 \rrbracket }}

\def\as{{ \mathrm{a.s.}  }}
\def\ed{\stackrel{{\mathcal{D}}}{=}}
\def\iid{\stackrel{{\mathrm{i.i.d}}}{\sim}}

\newcommand{\<}{\ensuremath{ \langle }}
\renewcommand{\>}{\ensuremath{ \rangle }}
\newcommand{\vect}[1]{\ensuremath{\boldsymbol{\mathbf{#1}}}}
\newcommand{\rdmvect}[1]{\ensuremath{\bm{#1}}}
\newcommand{\mat}[1]{\ensuremath{\boldsymbol{\mathbf{#1}}}}
\newcommand{\rdmmat}[1]{\ensuremath{\bm{#1}}}

\newcommand{\Tr}{\ensuremath{\mathrm{Tr}}}
\renewcommand{\det}{\ensuremath{\mathrm{det}}}

\newcommand{\norm}[1]{\ensuremath{ \|#1 \|}}

\newcommand{\matX}{\rdmmat{X}}
\newcommand{\matY}{\rdmmat{Y}}

\newcommand{\matXspk}{\rdmmat{\Tilde{X}}}
\newcommand{\matYspk}{\rdmmat{\Tilde{Y}}}

\newcommand{\matS}{\rdmmat{S}}
\newcommand{\matSx}{\rdmmat{S}_{x}}
\newcommand{\matSy}{\rdmmat{S}_{y}}
\newcommand{\matSspk}{\rdmmat{\Tilde{S}}}
\newcommand{\svdist}{\mu}

\newcommand{\girko}{\eta}
\newcommand{\matrixset}{\mathbb{M}}
\newcommand{\nbymmatrix}{\matrixset_{n,m}}
\newcommand{\samatrix}{\matrixset^{H}}

  \newcommand{\usx}{\vect{u}^\star_x}
  \newcommand{\usxk}{\vect{u}^{\star}_{x,k} }
   \newcommand{\vsx}{\vect{v}^{\star}_{x} }
  \newcommand{\vsxk}{\vect{v}^{\star}_{x,k} }
\newcommand{\usy}{\vect{u}^{\star}_{y} }
\newcommand{\usyk}{\vect{u}^{\star}_{y,k} }
\newcommand{\vsy}{\vect{v}^{\star}_{y} }
\newcommand{\vsyk}{\vect{v}^{\star}_{y,k} }

  \newcommand{\lyk}{\lambda_{y,k}} 
  \newcommand{\lxk}{\lambda_{x,k}}

  \newcommand{ \matQk}{ {\mat{Q}_k^{\star} }}
  \newcommand{ \matG}{ {\rdmmat{G}}}
  \newcommand{ \matGspk}{ {\rdmmat{\Tilde{G}} } }

      \newcommand{ \matKk}{\rdmmat{H}_{k,n} }
  \newcommand{ \matKspkk}{ {\rdmmat{\Tilde{H}}_{k,n} } }
  
\usepackage[authoryear,round]{natbib}

\begin{document}

%

%

\twocolumn[

\aistatstitle{Spectral Thresholds in Correlated Spiked Models and Fundamental Limits of Partial Least Squares}

\aistatsauthor{ Pierre Mergny \And Lenka Zdeborov\'a }

\aistatsaddress{ Information, Learning and Physics Laboratory\\
  École Polytechnique Fédérale de Lausanne \\
  1015 Lausanne, Suisse  \And  Statistical Physics of Computation Laboratory  \\
 École Polytechnique Fédérale de Lausanne  \\
 1015 Lausanne, Suisse} ]

\begin{abstract}
We provide a rigorous random matrix theory analysis of spiked cross-covariance models where the signals across two high-dimensional data channels are partially aligned. These models are motivated by multi-modal learning and form the standard generative setting underlying Partial Least Squares (PLS), a widely used yet theoretically underdeveloped method. We show that the leading singular values of the sample cross-covariance matrix undergo a Baik–Ben Arous–Péché (BBP)-type phase transition, and we characterize the precise thresholds for the emergence of informative components. Our results yield the first sharp asymptotic description of the signal recovery capabilities of PLS in this setting, revealing a fundamental performance gap between PLS and the Bayes-optimal estimator. In particular, we identify the SNR (signal-to-noise ratio) and correlation regimes where PLS fails to recover any signal, despite detectability being possible in principle. These findings clarify the theoretical limits of PLS and provide guidance for the design of reliable multi-modal inference methods in high dimensions. 
\end{abstract}

\section{INTRODUCTION}
\label{sec:intro}

The challenge of recovering a low-dimensional structure hidden in a high-dimensional noisy output is widespread in statistics, probability, and machine learning. Spiked random matrix models \cite{arous2005phase,johnstone2009consistency,lelarge2017fundamental,zou2006sparse}  have attracted significant attention as a simple  yet rich framework for studying this class of problems, especially within the toolbox of random matrix theory (RMT) \cite{AGZ,tao2012topics,potters2020first}, which provides asymptotic characterizations of spectral properties in high dimensions.  On the other hand, in more complex data scenarios, one often has access to multiple related outputs. Multi-modal learning \cite{ngiam2011multimodal,ramachandram2017deep}, a central paradigm in modern data analysis, seeks to leverage the joint information contained in such datasets to improve inference or prediction. This includes, for instance, settings where signals are observed across different modalities or sensors.

Popular classical approaches such as Canonical Correlation Analysis (CCA) \cite{thompson2000canonical,guo2019canonical,7903598} and Partial Least Squares (PLS) \cite{wold1975path,wold1983systems,wold2001pls,wegelin2000survey,pirouz2006overview} rely on spectral methods to uncover such cross-dependencies and have been widely applied across various scientific and engineering domains. While CCA has been extensively analyzed in the literature \cite{yang2022limiting,yang2022sample,guo2019canonical,bykhovskaya2023high,ma2023sample,bao2019canonical}, notably through the lens of RMT, methods such as PLS, which operate directly on the (empirical) cross-covariance matrix, even thought their widespread use, remain less well understood from a theoretical point of view.  

To address this issue, we focus on a setting involving two \emph{correlated spiked matrix models} (or \emph{"channels"}), defined as follows:
\begin{align}
 \label{eq:initial_model_X}
&  \matXspk = \matX + \sum_{k=1}^r \sqrt{\lxk} \cdot \usxk \, (\vsxk)^T \in \mathbb{R}^{n  \times d_x} \, ,  \\ 
\label{eq:initial_model_Y}   
& 
     \matYspk = \matY + \sum_{k=1}^r \sqrt{\lyk} \cdot \usyk \, (\vsyk)^T \in \mathbb{R}^{n \times d_y} \, .
\end{align}
The precise description and assumptions of this model are provided in more detail in the following paragraph. For now, the matrices $\matX$ and $\matY$ 
as \emph{noise matrices}. The  values $\lambda_{x,k}, \lambda_{y,k} \geq 0$ of the low-rank matrices serve as signal-to-noise ratios (SNR) for the components $\{ \usxk, \vsxk, \usyk, \vsyk \}_{k=1}^r$ that one aims to infer. The term \emph{correlated} refers to the assumption that the unit-norm signals in the shared dimension of the two sources exhibit partial alignment in the high-dimensional regime. Specifically, we consider that $
    \langle \usxk, \usyk \rangle \approx \rho_k \in (-1,1) $ for some fixed constants $\rho_k$.

Our objective is to provide a quantitative analysis of the Partial Least Squares (PLS) methods and derive their performance analytically to ease and inform the comparison with other estimators.  The PLS methods are based on estimating the subspace spanned by the signals based on the singular vectors of the \textbf{sample cross-covariance matrices}:
\begin{align}
\label{eq:def_matSspk}
    \matSspk &:= \matXspk^T \matYspk \in  \mathbb{R}^{d_x \times d_y} \, ,
\end{align}
and thus a high-dimensional setting where 
\begin{align*}
n, d_x, d_y  \to \infty \; \mbox{such that} \; \frac{n}{d_x} \to \alpha_x, \frac{n}{d_y} \to \alpha_y  \, , 
\end{align*}
for some positive constants $\alpha_x, \alpha_y$ and without loss of generality, we set $\alpha_x \geq \alpha_y$ (so $d_x \le d_y$). 

This model has recently been considered as a natural toy model for describing \emph{multi-modal learning} in high-dimensional settings, see \cite{abdelaleem2023simultaneous,keup2024optimal}.  The key intuition is that stronger alignment between the correlated signal vectors may facilitate the recovery of the other components. Authors of \cite{keup2024optimal}
 evaluated the detectability threshold of the Bayes-optimal estimator and compared it empirically with the ones of the PLS and CCA. While for the unimodal spiked matrix model, the detectability threshold of the Bayes-optimal estimator coincides with the BBP threshold for the natural spectral methods, the authors of \cite{keup2024optimal} pointed out based on numerical experiments that for the above correlated spiked matrix model the detectability threshold of both the PLS and CCA are suboptimal. This observation is interesting in particular in the view of the contrast with the unimodal case.

\paragraph{Main results --} In this work, we provide the high-dimensional limits of the spectral method based on the sample cross-covariance matrix given by Eq.~\eqref{eq:def_matSspk}. Specifically, 
\begin{itemize}
    \item We show that its leading singular values undergo a BBP-like phase transition depending on the value of the SNR, correlation and aspect ratios.  
    \item We obtained a complete characterization of the associated overlap with the hidden signal, and their phase transition, further generalizing the BBP results to this cross-product setting. 
    \item We apply this result to obtain the fundamental limits of Partial Least Squares (PLS) methods. 
    \item We discuss the comparison between the PLS, CCA and Bayes-optimal thresholds unveiling the somewhat surprising sub-optimality of these spectral approaches even on a model as simple as the one considered here.
\end{itemize}

\paragraph{Related work --} Many variants of spiked matrix models have been investigated in the high-dimensional regime; see \cite{peche2014deformed,arous2005phase,paul2007asymptotics,perry2018optimality,benaych2011eigenvalues,guionnet2023spectral} for general references, and in particular \cite{loubaton2011almost,capitaine2018limiting,benaych2012singular} for the case of spectral analysis in a \emph{single-channel} setting (also known as the \emph{information-plus-noise spiked model}) —namely, the study of the asymptotic behavior of the singular values and vectors of $\matXspk$ (or equivalently, of $\matYspk$). The properties of the spectrum of sample cross-covariance matrices \emph{without any low-rank
perturbation} are also well understood, see \cite{burda2010eigenvalues,akemann2013products}. To the best knowledge of the authors, this work is the first to characterize analytically the spectral properties of these cross-covariance matrices \emph{with the inclusion of such spikes}, answering in particular the question left open in \cite{keup2024optimal} where such models have been introduced as a multi-modal toy model and studied from a Bayes-Optimal (BO) point of view. As a direct application of our work, we obtain the fundamental limits of the performance of Partial Least Squares (PLS) methods. Despite the lack of previous theoretical guarantees (see \cite{keup2024optimal} and \cite{abdelaleem2023simultaneous} for an exploratory empirical study), PLS methods are widely used across a variety of fields \cite{hulland1999use,krishnan2011partial}. Closely related to PLS, CCA (sometimes refereed to as mode-B PLS) which operates on the MANOVA matrix $(\matXspk^T \matXspk)^{-1/2} \matXspk ^T \matYspk (\matYspk^T \matYspk)^{-1/2}$ rather than the cross-covariance matrix $\matSspk$ has been recently studied from a theoretical point of view in \cite{guo2019canonical,yang2022limiting,yang2022sample,bykhovskaya2023high} and our work compares the phase diagram of PLS methods with both CCA and BO methods. Similarly, several studies \cite{sergazinov2025spectral, lebeau2024performance, ma2023community,yang2024fundamental,duranthon2023optimal,tabanelli2025computationalthresholdsmultimodallearning,pacco2023overlaps,qi2025statistical} investigated the information theoretical limit for the inference of multiple spiked models with correlated signals, albeit in a different setting.  
Finally, we mention related work \cite{benaych2023optimal,attal2025eigenvector} which analyzes cross-covariance models in a regime where the number of spikes grows linearly with dimension, leading to fundamentally different spectral behavior.

Following the release of the first version of this work, \citet{leger2025high} independently reported related spectral and overlap results for high-dimensional PLS-SVD in a framework with joint and individual-specific components, and \citet{swainridoutnemenman26} studied the spectral properties of the stacked data matrix in a related setting.

This version restores a term omitted from the resolvent in the first version of the paper and updates the resulting formulas, with only minor numerical changes and no change in qualitative behavior.

\paragraph{Notations --}  In the following, generic vectors are written in bold lowercase (e.g. $\vect{v}$) while matrices are written in bold uppercase (e.g. $\mat{M}$). To further emphasize on the random nature of vectors (or matrices), we will use italics when needed (e.g. $\rdmvect{v}$ and $\rdmmat{M}$ respectively). In a similar way, we use the \emph{tilde}-notations (e.g. $\rdmvect{\Tilde{v}},\rdmmat{\Tilde{M}}$) to highlight the dependency in the low-rank signal components. For $k \in \mathbb{N}$, $\integset{k}= \{1,\dots,k\}$. $\mathbb{C}_{\pm} = \{ z \in \mathbb{C} | \pm \mathfrak{Im}(z) > 0  \}$ correspond to the upper and lower half complex plane. For $\mat{M} \in \mathbb{R}^{n \times m}$ with $n \leq m$, we denote by $\sigma_1(\mat{M}) \geq \dots \geq \sigma_n(\mat{M}) \geq 0$ its singular values and  $\svdist_{\mat{M}} := \frac{1}{n} \sum_{i=1}^n \delta_{\sigma_i(\mat{M})^2}$ the associated empirical \emph{squared} singular value distribution. We say that a sequence $(X_n)_n$ of random variables converges exponentially fast to a value $x$ if for all $\epsilon>0$, there exist $C,c>0$ such that $\mathbb{P}(|X_n - x|>\epsilon) \leq C \mathrm{e}^{-cn}$. $\mathcal{P}_c(\mathbb{R})$ denotes the set of compactly supported measures on $\mathbb{R}$. For a sequence $\mu_n \in \mathcal{P}_c(\mathbb{R})$, we denote $\mu_n \to \mu$ the (almost-sure) weak convergence. We denote by $\eta( \mat{A}) :=   \begin{pmatrix}
    0 & \mat{A} \\
    \mat{A}^T & 0
\end{pmatrix}$ the Hermitian embedding of $\mat{A}$.

\section{ASYMPTOTICS OF THE TOP SINGULAR VALUES AND OVERLAPS}

\subsection{Assumptions}
\label{sec:assumptions}
In what follows, we provide a more detailed description of the assumptions underlying our model for the cross-covariance matrix of Eq.~\eqref{eq:def_matSspk}. 

\textbf{Assumption (A1).} \emph{The  matrices $\matX= (X_{ij})_{i,j}$ and $\matY = (Y_{ij})_{i,j}$ are independent with entries given by $ \sqrt{d_x} \,  X_{ij} \iid \mathsf{N}(0,1)$ and  $ \sqrt{d_y}  \, Y_{ij} \iid \mathsf{N}(0,1)$.}

Although our analysis is carried out under this Gaussian assumption, universality results in RMT suggest that our conclusions should extend beyond this framework, typically requiring only  the entries to satisfy a log-Sobolev inequality, see for example \cite{baik2006eigenvalues} for the within-channel covariance matrix. We leave a rigorous investigation of this extension for future work. 

\textbf{Assumption (A2).} \emph{For $z \in \{x,y\}$ and any $k \in \integset{r}$, as $n \to \infty$ the low-rank signal components are independent of the noise matrices and satisfy $\norm{\vect{u}_{z,k}^\star} \xrightarrow[n\to \infty]{\as} 1$, $\norm{\vect{v}_{z,k}^\star} \xrightarrow[n\to \infty]{\as} 1$ and $ \langle \usxk , \usyk \rangle \xrightarrow[n\to \infty]{\as} \rho_k$ exponentially fast and all other inner product converge to zero exponentially fast.}

Assumption (A2) covers the case where the signal components are drawn independently from $\sqrt{d_x} \, \vsxk \sim \mathsf{P}^{\otimes d_x}_x$, $\sqrt{d_y} \, \vsyk \sim \mathsf{P}^{\otimes d_y}_y$ and $\sqrt{n}(\usxk, \usyk) \sim \mathsf{Q}^{\otimes n}_{(\rho_k)}$, where $\mathsf{P}_x$, $\mathsf{P}_y$ (resp. $\mathsf{Q}_{(\rho_k)}$) are probability measures on $\mathbb{R}$ (resp. on $\mathbb{R}^2$) with mean zero, variance one (resp. covariance {\tiny{$\begin{pmatrix}
    1 & \rho_k \\
    \rho_k &  1
\end{pmatrix}$}} ) satisfying standard log-Sobolev inequalities, as well as the cases where the signal components are deterministic vectors.   We point out that Assumption (A2) ensures that the matrix $(\vect{u}^\star_{x,1}, \dots , \vect{u}^\star_{x,r} )^T (\vect{u}^\star_{y,1},\dots ,\vect{u}^\star_{y,r})$ is asymptotically diagonal. In principle, the analysis developed in this paper could be extended beyond this setting, at the expense of greater technical complexity.

\subsection{Preliminary Definitions}
For $\mu \in \mathcal{P}_c(\mathbb{R}_+)$ and $z \in \mathbb{C} \setminus \mathrm{Supp}(\mu)$, we introduce the \emph{$T$-(moment generating) transform} as: 
\begin{align} 
\label{eq:def_Ttransform}
    t_\mu(z) := \int \frac{\lambda}{z-\lambda} \mathrm{d}\mu(\lambda) \, , 
\end{align}
which uniquely characterizes $\mu \in  \mathcal{P}_c(\mathbb{R}_+)$, see \cite{potters2020first}.

Next, for $z \in \mathbb{C}_-$ we introduce the cubic polynomial $P \in \mathbb{C}[X]$ given by 
\begin{align}
\label{eq:def_Pol_P}
    P_{(\alpha_x, \alpha_y)}(X,z)
    :=&
    1 + ( 1+ \alpha_x + \alpha_y -z)  X \notag \\
    &+ (\alpha_x + \alpha_y + \alpha_x \alpha_y)  X^2 + \alpha_x \alpha_y  X^3 \, .  
\end{align}
We define the following  cubic “edge”  polynomials $Q$ and  quadratic secular polynomial $R$ as
\begin{align}
\label{eq:def_Pol_Q}
    Q_{(\alpha_x,\alpha_y)}(X) 
    &:=
    1 -  (\alpha_x \alpha_y + \alpha_x + \alpha_y )   X^2 - 2 \alpha_x \alpha_y  X^3 \, , 
\end{align}
and
\begin{align}
\label{eq:def_Pol_R}
    R_{(\lambda_x,\lambda_y,\rho)}(X) :=& 1 - (\lambda_x + \lambda_y  + \rho^2 \lambda_x \lambda_y) X \notag \\ &+  ( 1 - \rho^2) \lambda_x \lambda_y X^2 \, ;
\end{align} 
and will use the following property. 
\begin{lem}[Roots of the edge and secular polynomials]
\label{lem:roots_polynomials}
The polynomial $Q_{(\alpha_x, \alpha_y)}$ has exactly one positive root which we denote by $\uptau^+ \equiv \uptau^+( \alpha_x, \alpha_y )$. Similarly for $\lambda_x,\lambda_y >0, |\rho|<1$ ,  the polynomial $ R_{(\lambda_x,\lambda_y,\rho)}$ has exactly two (counted with multiplicity) positive roots, which we denote by $\mathrm{r}^+ \equiv \mathrm{r}^+(\lambda_x,\lambda_y,\rho) \geq \mathrm{r}^- \equiv \mathrm{r}^-(\lambda_x,\lambda_y,\rho) >0$.   Furthermore, $\mathrm{r}^-$ (resp. $\mathrm{r}^+$) is a decreasing (respectively increasing) function of $\rho^2$. 
\end{lem}
\begin{proof}
    See Appendix \ref{sec:app:cubicroots}. 
\end{proof}

As we will show below, these roots (and more precisely their inverse) govern the BBP-like  phase transition of the model: the terms $1/\mathrm{r}^\pm(\lambda_x,\lambda_y,\rho)$ can be viewed as effective signal-to-noise ratios for the spiked cross-covariance matrix, while $1/\uptau^+$ plays the role of a limiting threshold for the appearance of spikes in the spectrum. In the following, to ease notations we write simply $\mathrm{r}_k^\pm \equiv \mathrm{r}^\pm(\lambda_{x,k},\lambda_{y,k},\rho_k)$.

\subsection{Bulk Distribution} 
We  recall the result regarding the behavior of the \emph{bulk} of the spectrum, from \cite{swain2025distribution,burda2010eigenvalues}. To this end, consider the constant  $\varsigma_+$ defined as
\begin{align}
\label{eq:def_edge}
      \varsigma_+ \equiv \varsigma_+(\alpha_x, \alpha_y)   :=  \sqrt{\frac{\big(1+\uptau^+\big)\big(1+\alpha_x \uptau^+\big) \big(1+\alpha_y \uptau^+\big)}{\uptau^+}} \, ,
\end{align}
which corresponds to the positive solution (in $z$) of the equation $P_{(\alpha_x,\alpha_y)}( \uptau^+,z^2) =0$. For notational convenience, we present the proposition under the assumption $\alpha_x > 1$, and refer the reader to Appendix~\ref{sec:app:bulk} for the complementary case.

\begin{prop}[Bulk Distribution]
\label{prop:bulk}
Under Assumptions (A1)-(A2) with $\matSspk$ given by Eq.\eqref{eq:def_matSspk},
\begin{itemize}
    \item[(i)]  we have $\mu_{\matSspk} \to \mu_{(\alpha_x,\alpha_y)}$ whose $T$-transform $t(z) \equiv t_{\mu_{(\alpha_x,\alpha_y)}}(z)$ defined by Eq.~\eqref{eq:def_Ttransform} is given for any $z \in \mathbb{C}_-$ as the analytical solution in $\mathbb{C}_+$ of
\begin{align}
\label{eq:fixedpoint_transform}
    P_{(\alpha_x,\alpha_y)} \bigg( \frac{t(z)}{\alpha_x},z \bigg) = 0 \, , \: t(z) \sim \frac{\alpha_x}{z} \; \mbox{as } |z|\to \infty;
\end{align}
    \item[(ii)] Furthermore, in the absence of planted signals $(\lxk = \lyk = 0$ for all $k \in \llbracket r \rrbracket$), the top singular value converges to squareroot of the rightmost edge of the squared value distribution $\mu_{(\alpha_x,\alpha_y)}$, given by the constant $\varsigma_+$ of Eq.~\eqref{eq:def_edge} and $\lim_{\epsilon \searrow0} t( \varsigma_+^2 + \epsilon)/\alpha_x = \uptau^+$ defined in Lemma~\ref{lem:roots_polynomials}. 
\end{itemize} 
\end{prop}
\begin{proof}
 This follows from free probability results (see \cite{mingo2017free} for an introduction to the topic). For completeness, the reader may find the proof in Appendix \ref{sec:app:bulk}.  
\end{proof}
Fig.~\ref{fig:sv_distribution} provides an illustration comparing the theoretical bulk of the spectrum with empirical histograms of the singular values for large but finite dimensions. 

\begin{figure*}[t]
    \centering
    \begin{subfigure}[t]{0.49\textwidth}
        \centering
\includegraphics[width=\textwidth]{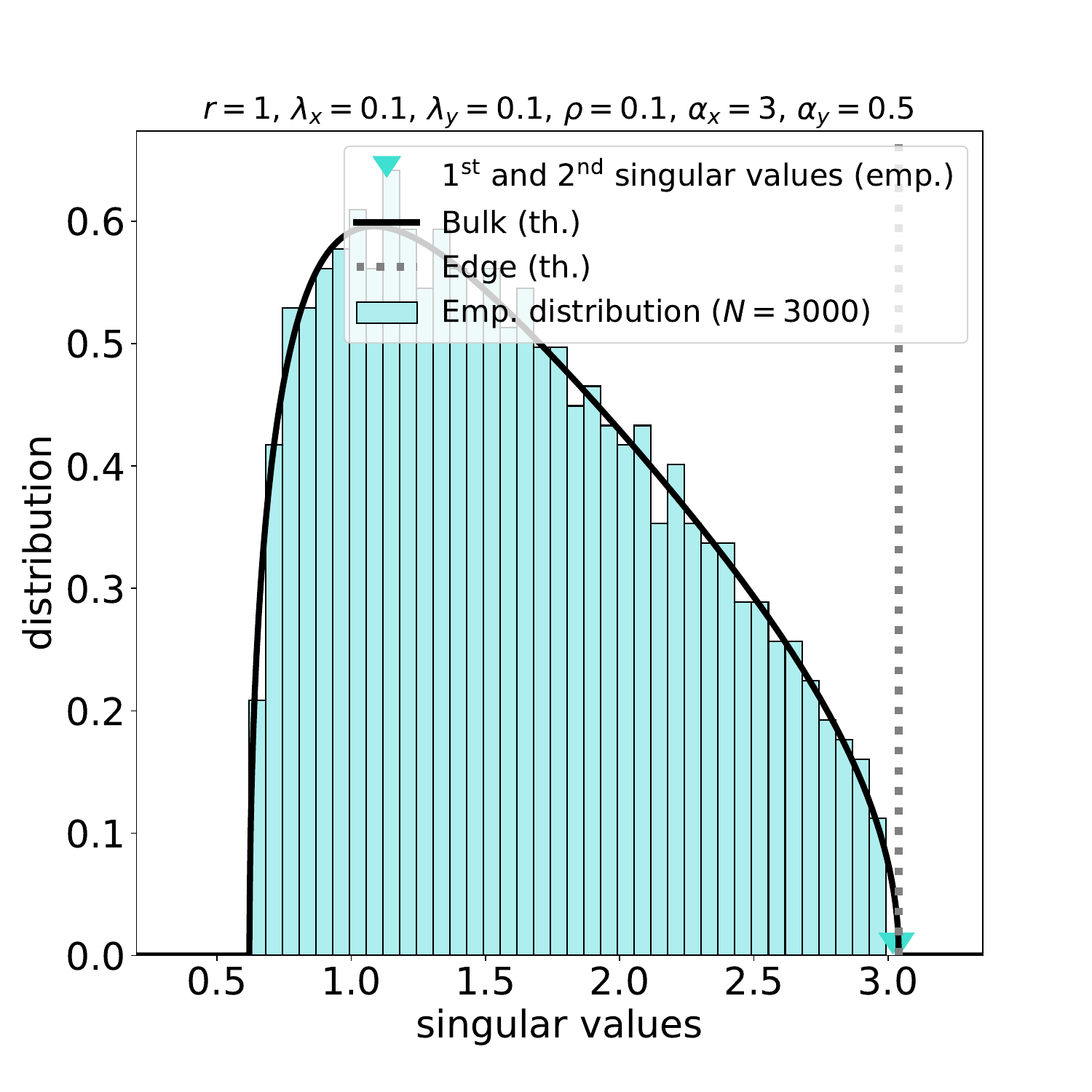}
        \label{fig:subfig_0spk}
    \end{subfigure}
    \hfill
    \begin{subfigure}[t]{0.49\textwidth}
        \centering
        \includegraphics[width=\textwidth]{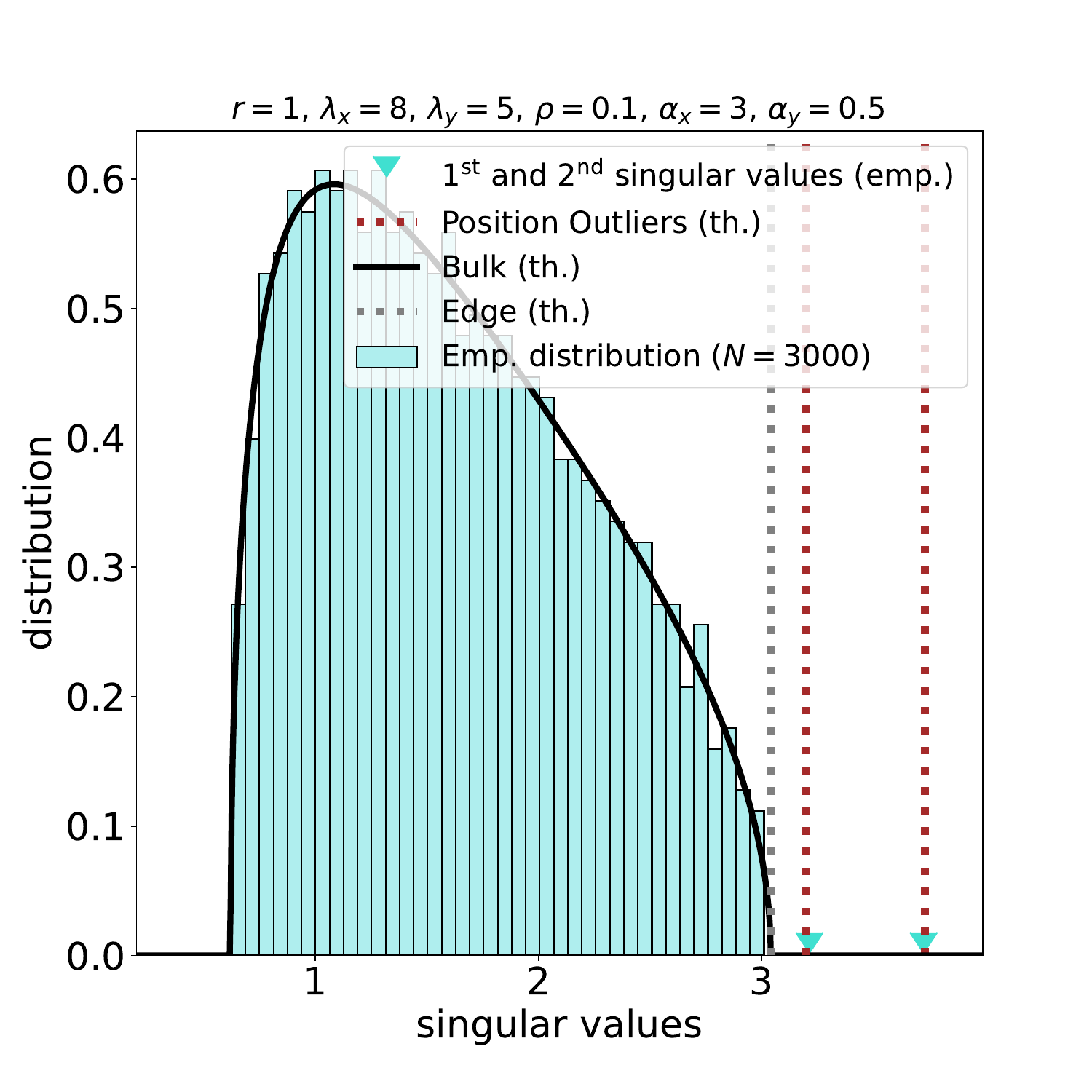}
        \label{fig:subfig_2spk}
    \end{subfigure}
    \caption{Singular values of the spiked cross-covariance \eqref{eq:def_matSspk} with one spike ($r=1$) in each channel,  for low (left) and high (right) values of the signal strengths  ($\lambda_x, \lambda_y$) while the other parameters ($\alpha_x,\alpha_y,\rho$) are the same. For high signal strengths, outliers separate from the bulk. For low signal strengths, the top two singular values stick to the bulk edge of the distribution defined in Prop.~\ref{prop:bulk}. The theoretical positions of the outliers for the right panel follow from our main result Thm.~\ref{thm:singularvalue}.   }
    \label{fig:sv_distribution}
\end{figure*}

\subsection{Asymptotics of the Top Singular Values}

Our first result characterizes how Part-(ii) of this proposition is modified by the presence of low-rank signals leading to a BBP-like phase transition phenomenon. To this end, we introduce the non-increasing function $ b: \mathbb{R}^+ \ni r \mapsto b(r) \in \mathbb{R}^+$ defined by
\begin{align}
\label{eq:def_bbpfunction}
b(r) 
:=     
\begin{cases}
\sqrt{\frac{\big(1+r \big)\big(1+\alpha_x r \big) \big(1+\alpha_y r \big)}{r}} & \mbox{if } 1/r \geq 1/\uptau^+ \, ,\\
    \\
        \varsigma_+ & \mbox{otherwise} \, . 
    \end{cases}
\end{align}
which appears in the following result.

\begin{theo}[Singular Values]
\label{thm:singularvalue}
Under Assumptions (A1)-(A2), the top $2 r$ singular values of the matrix $\matSspk$ given by Eq.~\eqref{eq:def_matSspk} converges almost-surely (up to re-ordering) to
\begin{align}
\{ \mathrm{\sigma}_1( \matSspk), \dots ,   \mathrm{\sigma}_{2r}( \matSspk)\}\xrightarrow[n \to \infty]{\as}
\{  b( \mathrm{r}_k^+)           \;  , \;  b( \mathrm{r}_k^-)  \}_{k=1}^r   \, , 
\end{align}
where the $\mathrm{r}_k^\pm \equiv \mathrm{r}^\pm(\lambda_{x,k},\lambda_{y,k},\rho_k)$ are given in Lemma~\ref{lem:roots_polynomials}. 
\end{theo}
\begin{proof}
    The proof of this Theorem is detailed in Sec.~\ref{sec:proof_singularvalues}. 
\end{proof}
Note that, by convention, $\mathrm{r}_k^- \leq \mathrm{r}_k^+$. Since $b$ is non-increasing, $b(\mathrm{r}_k^+) \leq b(\mathrm{r}_k^-)$. Thus, for each component $k$, the minus superscript labels the larger of its two limiting singular values. Thm.~\ref{thm:singularvalue} identifies the threshold at which the first outlier separates from the bulk, thus distinguishing the spiked model from the purely noisy one, as the solution to an implicit equation involving the roots of the cubic edge polynomial $Q$ and the quadratic polynomial $R$, described in Lemma~\ref{lem:roots_polynomials}. Specifically, the model is said to be at criticality, meaning that the planted signal is just strong enough to produce a detectable spectral spike, when 
\begin{align}
\label{eq:threshold}
    \mathrm{min}_{k \in \integset{r}} \{ \mathrm{r}^-( \lambda_{x,k},\lambda_{y,k},\rho_k) \} = \uptau^+( \alpha_x, \alpha_y) \, . 
\end{align}
By Lemma~\ref{lem:roots_polynomials}, $ \mathrm{r}^- $ is a decreasing function of $ \rho^2 $, the squared correlation between the latent variables, and $ b(\cdot) $ is a non-increasing function of its argument. Thus, if we denote by $k_0$ the argmin in Eq.~\eqref{eq:threshold}, as the squared correlation increases while keeping other parameters fixed, the leading outlier at $b(\mathrm{r}_{k_0}^-)$ moves further away from the edge of the bulk, thereby improving detectability. Note that, conversely, by Lemma~\ref{lem:roots_polynomials} the outlier at $b(\mathrm{r}_{k_0}^+)$ moves towards the rightmost edge. In Fig.~\ref{fig:figure_BBP_sv.pdf}, we have illustrated the theoretical prediction given by Thm.~\ref{thm:singularvalue} against empirical simulations at finite but large dimensions, showing good agreement. Further results on how the singular values depend on the model parameters are discussed in App.~\ref{sec:formula_outlier_overlap}. 

The phase diagram in the $(\lambda_{x}, \lambda_y)$-plane for the spectral methods based on the cross-covariance matrix (and hence of PLS as well) is illustrated in Fig.~\ref{fig:figure_phasediagram}. In the same figures, we have illustrated the threshold for PCA on the single-channel spiked models, the CCA matrix \cite{ma2023sample,bykhovskaya2023high} and the Bayes approach \cite{keup2024optimal}. In particular, although the parameter region for detectability expands as the squared correlation $ \rho^2 $ increases, we note that for small but non-zero values of $ \rho^2 $, it is possible to enter a regime where detection is achievable via PCA on the \emph{single-channel}, but \emph{not} on the cross-covariance matrix. This raises questions about the practical advantage of using the cross-covariance matrix and PLS approach in such scenarios, despite their widespread use.

\begin{figure*}[t]
    \centering
    \begin{subfigure}[t]{0.42\textwidth}
        \centering
\includegraphics[width=\textwidth]{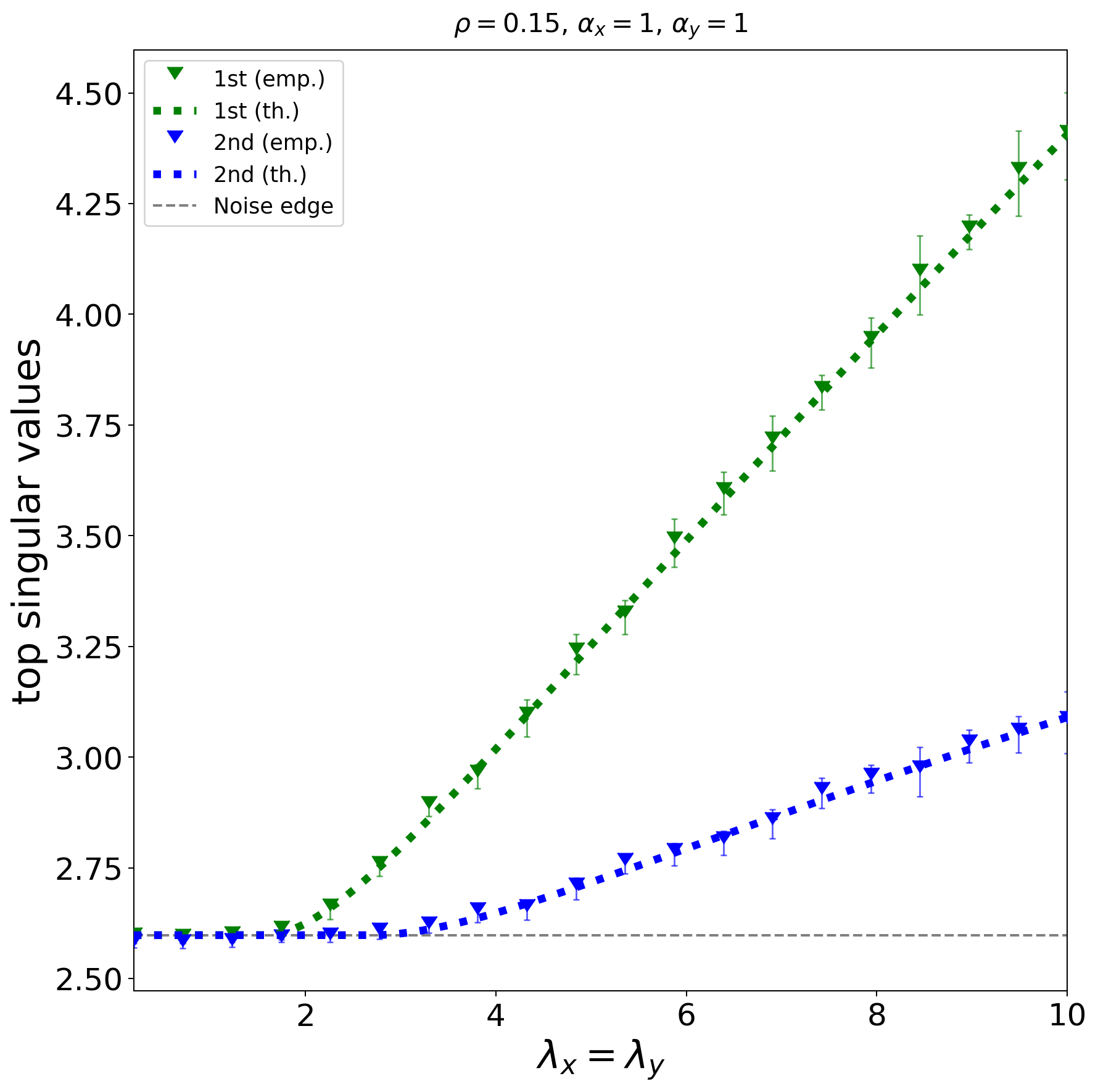}
        \label{fig_BBP_sv_v2}
    \end{subfigure}
    \hfill
    \begin{subfigure}[t]{0.42\textwidth}
        \centering
        \includegraphics[width=\textwidth]{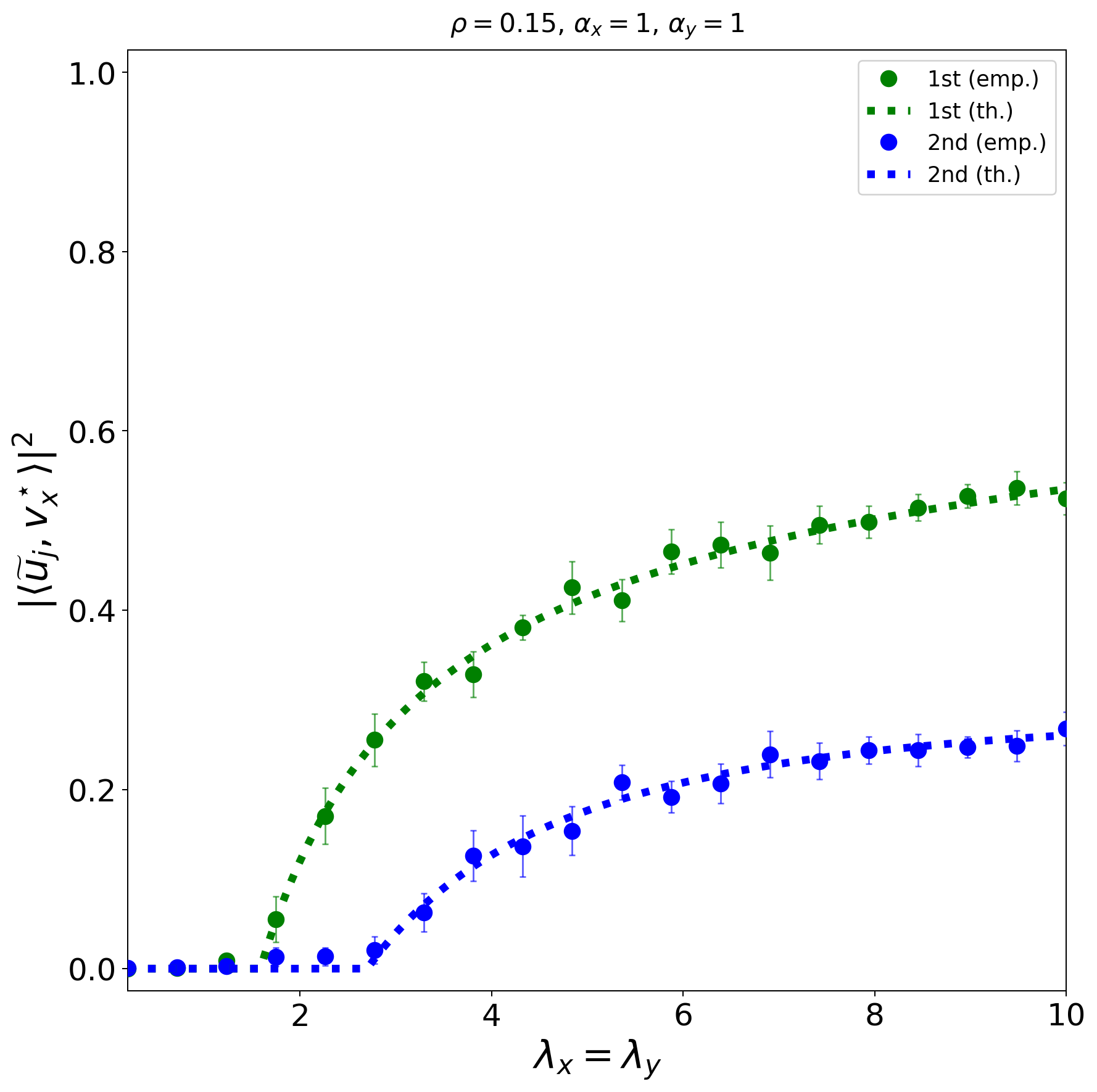}
        \label{fig_BBP_ov_v2}
    \end{subfigure}
    \caption{ Empirical and theoretical values of the top
two singular values (left) and the corresponding overlaps (right) of the spiked cross-covariance matrix with one spike $(r=1)$, plotted as functions
of the signal strengths $\lambda_x = \lambda_y$. All other parameters $\alpha_x ,\alpha_y, \rho$ are fixed. ($n=3000$, average over $10$ samples per point) }
    \label{fig:figure_BBP_sv.pdf}
\end{figure*}

\subsection{Asymptotics of the Top Singular Vectors}
We now turn to our second main result, which analyzes the overlaps between the singular vectors and the signal components. Let us introduce
\begin{align}
\label{eq:def_overlapfunction}
m_{\gamma,a}(r) 
:=     
\begin{cases}
 \frac{(\gamma r -1) Q_{(\alpha_x,\alpha_y)}(r) }{r(1+r)(1 + a r) R_{(\lambda_x,\lambda_y,\rho)}^{\prime}(r) }
 & 
 \mbox{if } 1/r \geq 1/\uptau^+ \, ,\\\\
    0 
    & 
    \mbox{otherwise} \, ,
    \end{cases}
\end{align} 
Evaluating this function at the roots of $R$ with $\gamma = \lambda_{\bar{z}}$, $a = \alpha_{\bar{z}}$, where we use $\bar{x} =  y$ and $\bar{y} =x$ to denote the opposite channel.

\begin{theo}[Singular Vectors]
\label{thm:overlap}
For $k \in \llbracket r \rrbracket$, let $ (\rdmvect{\Tilde{u}}^+_k,  \rdmvect{\Tilde{v}}^+_k)$ and $ (\rdmvect{\Tilde{u}}^-_k,  \rdmvect{\Tilde{v}}^-_k)$ denote the unit-norm singular vectors of $\matSspk$ associated with the two singular values whose limiting values are given by $b(\mathrm{r}_k^+)$ and $b(\mathrm{r}_k^-)$ respectively, as stated in Thm.~\ref{thm:singularvalue}, and are simple above the transtion. Then for any $k \in \llbracket r \rrbracket$ we have :
\begin{align*}
     \langle \rdmvect{\Tilde{u}}^\pm_k, \vect{v}^{\star}_{x,k} \rangle^2 
     &\xrightarrow[n \to \infty]{\as} 
    m_{\lambda_{y,k},\alpha_{y}}( \mathrm{r}_k^\pm ) ,
    \\
    \langle \rdmvect{\Tilde{v}}^\pm_k, \vect{v}^{\star}_{y,k} \rangle^2 
    &\xrightarrow[n \to \infty]{\as} 
    m_{\lambda_{x,k},\alpha_{x}}( \mathrm{r}_k^\pm ). 
\end{align*}
\end{theo}
\begin{proof}
    The proof of this Theorem is detailed in Sec.~\ref{sec:proof_singularvectors}. 
\end{proof}
Additional properties of the behavior of the singular vectors overlaps with respect to the parameters are discussed in App.~\ref{sec:formula_outlier_overlap}. 

\noindent\textbf{Remark}{ (Suboptimal Rotation)}\textbf{.} This theorem implies that when $1/\mathrm{r}_k^{+} >1/\uptau^+$, the two left (resp. right) singular vectors associated with $b(\mathrm{r}^-_k)$ and $b(\mathrm{r}^+_k)$ \textbf{both} correlate with $\vect{v}^\star_{x,k}$ (resp. with $\vect{v}^\star_{y,k}$). A more accurate estimation can be achieved by appropriately \emph{rotating} these two vectors, see Appendix \ref{sec:app:angle_rotation}.

\section{APPLICATIONS TO THE FUNDAMENTAL LIMITS OF PLS METHODS}
Following the description in \cite{keup2024optimal} (see also \cite{wegelin2000survey}), the canonical (or "mode-A") PLS algorithm with $r_0$ steps, constructs rank-$r_0$ approximations of $\matXspk$ and $\matYspk$ in order to estimate the underlying low-rank structures. This is achieved by iterating the following five steps $r_0$ times:

\begin{itemize}
    \item[(1)] compute the top left and right singular vectors $\rdmvect{\Tilde{u}}$, $\rdmvect{\Tilde{v}}$ associated with the top singular value of $\matSspk$;
    \item[(2)] estimate the left (‘$u$') components by $ \widehat{\rdmvect{u}}^{(\mathrm{PLS})}_{x} = \matXspk \rdmvect{\Tilde{u}}$ and $\widehat{\rdmvect{u}}^{(\mathrm{PLS})}_{y} = \matYspk \rdmvect{\Tilde{v}}$;
    \item[(3)] refine the estimates of the right (‘$v$') components by $ \widehat{\rdmvect{v}}^{(\mathrm{PLS})}_{x} =  \|  \widehat{\rdmvect{u}}^{(\mathrm{PLS})}_{x} \|^{-2}  \matXspk^T \matXspk  \rdmvect{\Tilde{u}}$ and  $ \widehat{\rdmvect{v}}^{(\mathrm{PLS})}_{y} =  \|  \widehat{\rdmvect{u}}^{(\mathrm{PLS})}_{y} \|^{-2}  \matYspk^T \matYspk  \rdmvect{\Tilde{v}}$.
    \item[(4)] subtract the resulting rank-one approximations from each data matrix individually, updating $\matXspk \leftarrow \matXspk -\widehat{\rdmvect{u}}^{(\mathrm{PLS})}_{x} (\widehat{\rdmvect{v}}^{(\mathrm{PLS})}_{x})^T$ and $\matYspk \leftarrow \matYspk -\widehat{\rdmvect{u}}^{(\mathrm{PLS})}_{y} (\widehat{\rdmvect{v}}^{(\mathrm{PLS})}_{y})^T$;
    \item[(5)] repeat from step (1).
\end{itemize}
A simplified variant, known as \textbf{PLS-SVD}, omits step~(3), reducing PLS to a Lanczos-type procedure (a rank-$r_0$ power method), where the right components $v$ are estimated directly from the top $r_0$ singular vectors of the cross-covariance matrix.

By construction, both variants rely on the spectral properties of the cross-covariance matrix; their weak recovery thresholds, the point at which the estimates start to correlate with the underlying signals, are expected to be identical and corresponds to the point where the top singular vectors acquire non-zero overlap with the planted signals. According to Thm.~\ref{thm:overlap}, it coincides with the emergence of the first spectral outlier and is explicitly characterized by Eq.~\eqref{eq:threshold}. For one planted direction per channel ($r=1$), this gives the following corollary.

\begin{cor}
Consider the channels in Eq.~\eqref{eq:initial_model_X} and Eq.~\eqref{eq:initial_model_Y} with $r=1$. For $z \in \{x,y\}$, let $\widehat{\rdmvect{v}}^{(\mathrm{PLS})}_{z}$ denote the estimator obtained from either of the two PLS variants described above with one ($r_0=1$) step. Then,
\begin{align*}
    \big\langle \widehat{\rdmvect{v}}^{(\mathrm{PLS})}_{z}, \vect{v}^{\star}_{z} \big\rangle^2 
    \xrightarrow[n \to \infty]{\as} 
    \mathrm{m}^{(\mathrm{PLS})}_{v,z},
\end{align*}
where $\mathrm{m}^{(\mathrm{PLS})}_{v,z}>0$ if and only if $1/\mathrm{r}^- > 1/\uptau^+$, and $\mathrm{m}^{(\mathrm{PLS})}_{v,z}=0$ otherwise. An analogous statement holds for the estimators $\widehat{\rdmvect{u}}^{(\mathrm{PLS})}_{z}$.
\end{cor}

For comparison, CCA is based on the singular values of the whitened cross-covariance matrix
$\bm{C}_{x,y} = (\matXspk^T\matXspk)^{-1/2}\matSspk
    (\matYspk^T\matYspk)^{-1/2}$. For $r=1$, and our  model, the results of \cite{ma2023sample,bykhovskaya2023high} writes for $1/\alpha_x+1/\alpha_y<1$, that the top singular value of $\bm{C}_{x,y}$ separates from the noise bulk if and only if $\theta_{\mathrm{CCA}}>\theta_{\mathrm{c}}$ with
\begin{align}
\theta_{\mathrm{CCA}}
&:=\frac{\rho^2\lambda_x\lambda_y}{(\alpha_x+\lambda_x)(\alpha_y+\lambda_y)}
, 
\theta_{\mathrm{c}}
:=\frac1{\sqrt{(\alpha_x-1)(\alpha_y-1)}}.
\end{align}
Here $\theta_{\mathrm{CCA}}$ is the squared population canonical correlation. The following result compares this threshold with the PLS threshold.

\begin{prop}[PLS dominates CCA]
\label{prop:PLS_vs_CCA}
Assume $r=1$, $\lambda_x,\lambda_y>0$, $|\rho|\leq1$ and $1/\alpha_x+1/\alpha_y<1$, we have $\theta_{\mathrm{CCA}}>\theta_{\mathrm{c}}
\Rightarrow
1/\mathrm{r}^->1/\uptau^+$.
\end{prop}
\begin{proof}
    See Appendix~\ref{sec:PLS_VS_CCA}.
\end{proof}
\noindent Thus, the region where CCA has a separated outlier is strictly contained in the corresponding region for PLS.

\begin{figure*}[t]
    \centering
    \begin{subfigure}[t]{0.42\textwidth}
        \centering
\includegraphics[width=\textwidth]{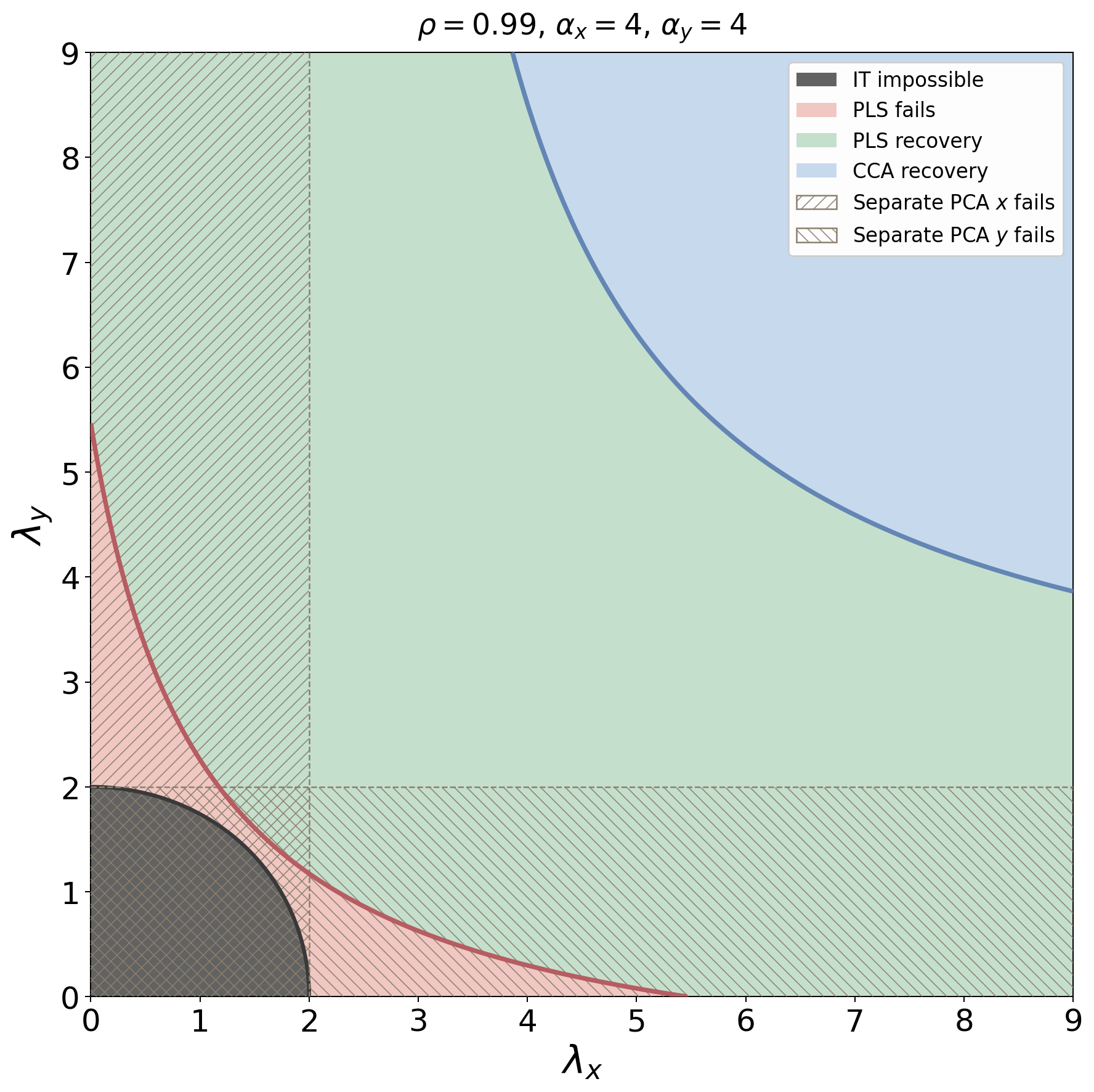}
        \label{fig_DP_v2}
    \end{subfigure}
    \hfill
    \begin{subfigure}[t]{0.42\textwidth}
        \centering
        \includegraphics[width=\textwidth]{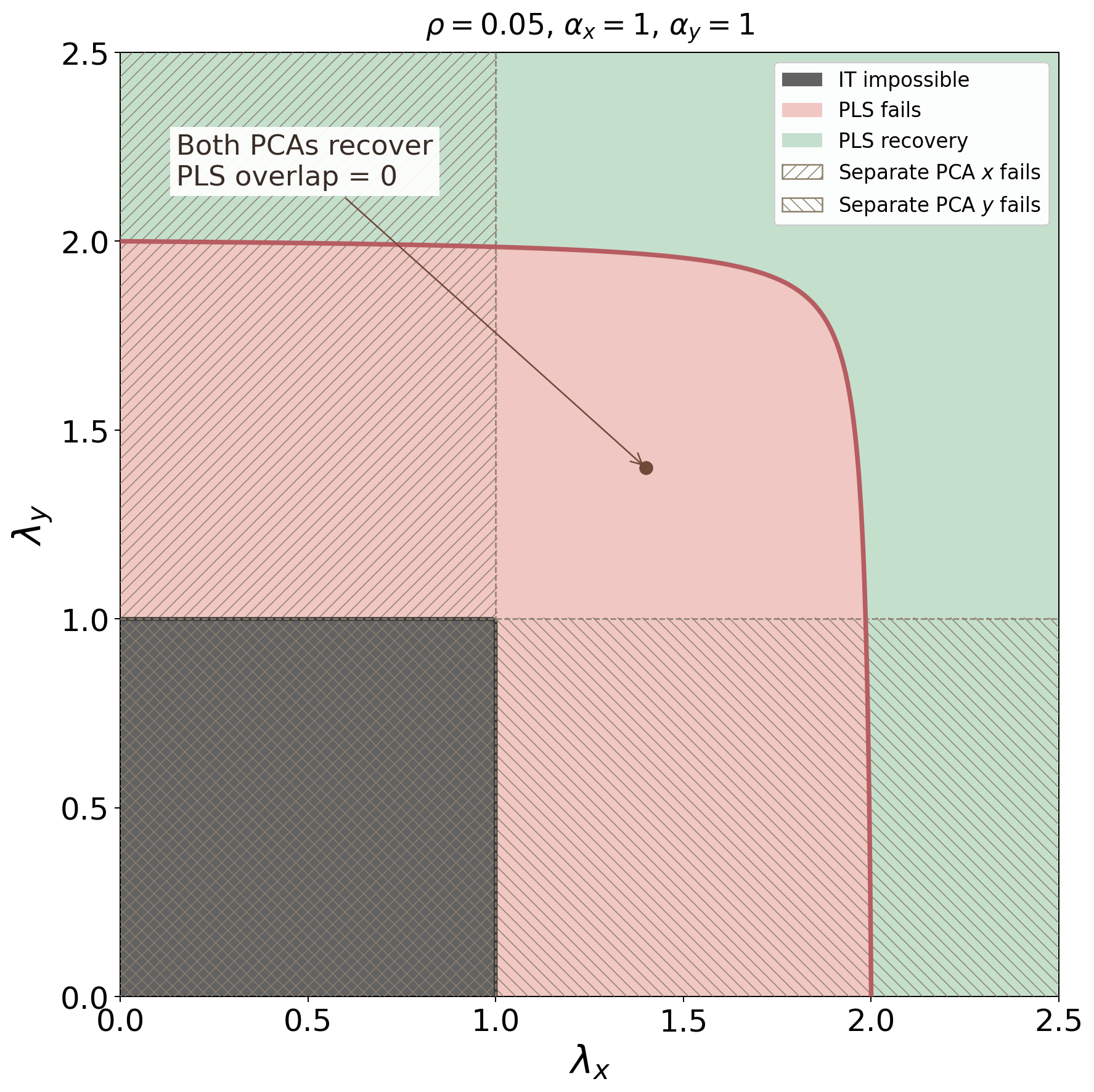}
        \label{fig_DP_lowcorrelation}
    \end{subfigure}
\caption{Phase diagram in the $(\lambda_x,\lambda_y)$ plane for high (left) and low (right) correlation $\rho$ illustrating the detection thresholds for Bayes-optimal method \cite{keup2024optimal} (black), single-view channel SVD (gray), PLS (this paper, red), and CCA \cite{bykhovskaya2023high,ma2023sample} (blue) algorithms. The region below each curve corresponds to values of the signal strengths for which spike detection is impossible. The right panel illustrates a broad range of signal strength for  which PLS cannot recover the signals while both separate PCA can.   }
    \label{fig:figure_phasediagram}
\end{figure*}

 This behavior is illustrated in Fig.~\ref{fig:figure_phasediagram}, where we compare it with the thresholds achieved by other methods (CCA, single-channel spectral methods, and the Bayes-optimal benchmark) and in particular show that \emph{PLS is most effective in asymmetric regimes} where one data channel carries a strong, clearly detectable signal while the other contains a weak signal. When the two are sufficiently correlated, the strong channel can ``lift'' the weak one, allowing joint recovery that would not be possible from the weak channel alone. Above this threshold, PLS achieves non-zero overlap with both planted directions, thereby exploiting the cross-channel correlation to transfer information. This leveraging effect is visible in the phase diagram (bottom-right region of Fig.~\ref{fig:figure_phasediagram}), where even if one signal lies below its single-channel detectability threshold, PLS succeeds provided the correlation with the strong channel is large enough. By contrast, \emph{when both signals are weak} (bottom-left region of Fig.~\ref{fig:figure_phasediagram}), \emph{PLS offers no advantage over single-channel spectral methods} (based on the within-channel covariance matrices) and may even  underperform them. The Bayes-optimal estimator combines the best of both worlds: it constructs a block matrix that integrates the cross-covariance and the within-channel covariance matrices, with weights optimally chosen as functions of the model parameters and described explicitly in~\cite{keup2024optimal}.

The performance of PLS when extracting multiple components ($r_0>1$) in the presence of several latent directions in each channel ($r>1$) can, in principle,  also be inferred directly from Theorem~\ref{thm:overlap} by successively projecting onto the corresponding signal subspaces while this extension is technically straightforward, it yields rather cumbersome expressions. Yet this analysis is more delicate than in the rank-one case. The difficulty arises from the fact that the \emph{ limiting positions of the outlier singular values,} $\{ b(\mathrm{r}_k^+), b(\mathrm{r}_k^-)\}_{k=1}^r$, are \emph{not necessarily ordered}. For instance, one can encounter configurations where $b(\mathrm{r}_{k_1}^+) > b(\mathrm{r}_{k_2}^-)$ for some $k_1 \neq k_2$, see Fig.~\ref{fig:pls_misordering} for an illustration.  As a result, the leading $r$ empirical singular values can include two from one component and none from another. PLS-SVD can therefore align with the ``wrong'' spike, that is, the algorithm produces an estimator correlated with the subspace spanned by $(\vect{v}^\star_{x,k_1},\vect{v}^\star_{y,k_1})$ instead of the intended $(\vect{v}^\star_{x,k_2},\vect{v}^\star_{y,k_2})$. Although detection thresholds remain the same for PLS and PLS-SVD, the quality of alignment with the true signal subspace may be compromised when multiple spikes interact.

\begin{figure}[h]
    \centering
    \includegraphics[width=0.9\linewidth]{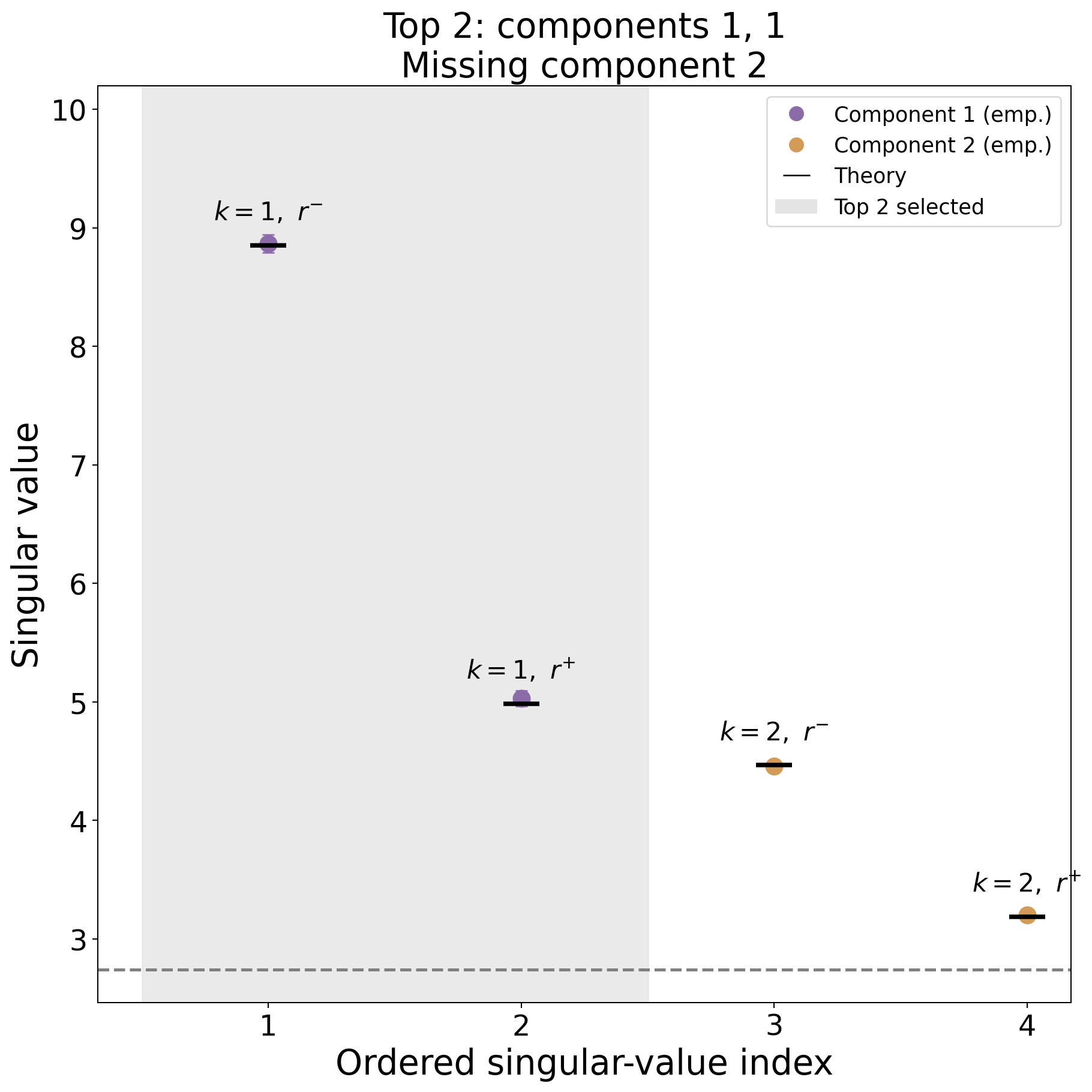}
    \caption{
    Misordering of PLS-SVD components in a rank-two model.
    Both leading singular values originate from component $1$,
    while component $2$ appears at ranks $3$ and $4$ ($n=3000$, 10 realizations).
    Parameters are $(\alpha_x,\alpha_y)=(2,1/2)$,
    $(\lambda_{x,1},\lambda_{y,1},\rho_1)=(40,40,0.1)$, and
    $(\lambda_{x,2},\lambda_{y,2},\rho_2)=(10,10,0.15)$.}
    \label{fig:pls_misordering}
\end{figure}

\section{OUTLINE OF THE PROOFS}
\label{sec:outline_proof}
\subsection{Proof of the Phase Transition for the Top Singular Values}
\label{sec:proof_singularvalues}
We introduce the \emph{resolvents} as
\begin{align}
\label{eq:resolvents}
    \rdmmat{G}(z) &= ( z - \girko(\matS) )^{-1} \, \quad 
    \mbox{and}
    \quad
    \rdmmat{\Tilde{G}}(z) = ( z - \girko(\matSspk) )^{-1} \, ,
\end{align}
and to ease notation, we also introduce for $k \in \llbracket r \rrbracket $ the matrices 
\begin{align}
\label{eq:projections}
    \mat{Q}^{\star}_k &:=
    \begin{pmatrix}
        \vsxk & \matX^T \usyk & \vect{0} & \vect{0} \\
        \vect{0} & \vect{0} & \vsyk & \matY^{T}\usxk  
    \end{pmatrix} \, , 
\end{align}
and the restrictions of the resolvents into the subspace generated by this matrix: 
\begin{align}
\label{eq:def_matK}
    \matKk(z) &:= \matQk^T  \matG(z) \matQk 
    \;
    \mbox{,}
    \; 
    \matKspkk(z):= \matQk^T \matGspk(z) \matQk  \, .
\end{align}
Both the resolvents and their restrictions into the subspace generated by the $\bm{Q}^\star_k$  will naturally appear in the proof.  We begin with an intermediate result characterizing the possible locations of the outliers. 
\begin{lem}
\label{lem:condition_outlier}
The singular values  of $\matSspk$ that are not the singular values of $\matS = \matX^T\matY$ are given by the positive solutions (in $z$) of the equation 
    \begin{align}
    \prod_{k=1}^r \mathrm{det} \left( 
    \matKk(z)
    -  \mat{A}_k
    \right) + \epsilon_n(z) = 0 
\end{align}
where $\epsilon_n(z) \xrightarrow[n\to \infty]{\as} 0 $ locally uniformly away from the support of the limiting spectral distribution and 
\begin{align}
   \mat{A}_k &:= \begin{pmatrix}
        0 & 0 & 0 & \lambda_{x,k}^{-1/2}   \\
        0 & 0 &  \lambda_{y,k}^{-1/2} & - \rho_{k} \\
        0 & \lambda_{y,k}^{-1/2} & 0 & 0 \\
        \lambda_{x,k}^{-1/2} & - \rho_{k}  & 0 &0 
    \end{pmatrix} \, . 
\end{align}
\end{lem}
\begin{proof}
    The proof follows from the determinant formula and is detailed in App. \ref{sec:app:proof_lem_condition}. 
\end{proof}

We then show that, in the high-dimensional regime $d_x,d_y,n\to\infty$, the high-dimensional random quantities entering the previous lemma admit low-dimensional deterministic equivalents. 
\begin{lem}
\label{lem:concentration_diagonal4by4}
For any $k \in \integset{r}$, the entries of the $(4 \times 4)$ matrix $\matKk(z)  = \big(H_{ij,n}(z)\big)_{1 \leq i,j \leq 4}$ of the previous lemma are given asymptotically by 
\begin{enumerate}
    \item[(i)]  $H_{11,n}(z)
    \xrightarrow[n \to \infty]{\as} h_1(z):=  \frac{t(z^2)+1}{z}\; $ ,
    \item[(ii)] $H_{22,n}(z)  \xrightarrow[n \to \infty]{\as}  h_2(z) := \frac{ z \cdot t (z^2) }{ \alpha_x + \alpha_y \cdot t (z^2)} \; $,
    \item [(iii)] $\; H_{33,n}(z)  \xrightarrow[n \to \infty]{\as} h_3(z) := \frac{ \alpha_y/\alpha_x t(z^2) +1}{z}$ ,
    \item[(iv)]$\; H_{44,n}(z) \xrightarrow[n \to \infty]{\as}  h_4(z) :=  \frac{1}{\alpha_x} \cdot 
       \frac{ z \cdot  t(z^2) }{1+  t(z^2)} $,
    \item[(v)] $H_{24,n}(z), H_{42,n}(z) \to \frac{\rho_k t(z^2)}{\alpha_x}$;
\end{enumerate}
and all other off-diagonal terms go to zero.
\end{lem}
\begin{proof}
    The details of the proofs are given in App.~\ref{sec:app:proof_lem_concentration}. They follow from concentration results for Part-(i) and the use of free probability theory for Part-(ii).  
\end{proof}
Using the limits of Lemma~\ref{lem:concentration_diagonal4by4}  in Lemma~\ref{lem:condition_outlier}, one ends up with $z$ being an outlier outside the bulk if and only if it is a (positive) zero of one of the functions 
    \begin{align*}
    j_k(z) = \mathrm{det} 
     \begin{pmatrix}
        h_1(z) & 0 & 0 & -\lambda_{x,k}^{-1/2}   \\
        0 & h_2(z) &  -\lambda_{y,k}^{-1/2} & h_5(z) \\
        0 & -\lambda_{y,k}^{-1/2} & h_3(z) & 0 \\
       - \lambda_{x,k}^{-1/2} &  h_5(z)  & 0 & h_4(z)
    \end{pmatrix}
\end{align*}
for $k \in \integset{r}$ and 
with 
\begin{align}
    h_5(z) := \rho_k \big( 1+ \frac{t(z^2)}{\alpha_x}\big). 
\end{align}
Substituting the expressions for $h_1,\dots,h_5$ gives writes after simplification: 
\begin{align}
\label{eq:jk_ttransform}
    j_k(z) =&
        \left( \frac{ t(z^2)}{\alpha_x} - \frac{1}{\lambda_{x,k}} \right)  \left( \frac{ t(z^2)}{\alpha_x} - \frac{1}{\lambda_{y,k}} \right)  \notag \\
        & - \rho_k^2 \frac{t(z^2)}{\alpha_x}\big(1 + \frac{t(z^2)}{\alpha_x} \big) , \\
        &= \frac{R_{(\lambda_{x,k},\lambda_{y,k},\rho_k)}(t(z^2)/\alpha_x)}{\lambda_{x,k} \lambda_{y,k}}
\end{align}
for $k \in \integset{r}$. From Prop.~\ref{prop:bulk}, $z$ and $t(z^2)$ are related by $P_{(\alpha_x,\alpha_y)}(t(z^2)/\alpha_x,z^2) =0$ which further implies that $z$ is the limiting value of an outlier if it satisfies 
\begin{align}
\label{eq:rootsR}
    R_{(\lxk,\lyk,\rho_k)}(t(z^2)/\alpha_x) = 0
\end{align}
for some $k \in \integset{r}$. Since $z \mapsto t(z^2)/\alpha_x$ is a continuous decreasing function in the interval $(\varsigma_+, \infty)$ with maximal value obtained at $\varsigma_+$ and given by $\uptau^+$, for the existence of the roots in Eq.~\eqref{eq:rootsR} to hold, one must have that the smallest positive root $\mathrm{r}^-$ of the quadratic polynomial $R_{(\lxk,\lyk,\rho_k)}$ is lower than this value $\uptau^+$, otherwise Eq.~\eqref{eq:rootsR} is never satisfied by the $T$-transform. Conversely, if $\uptau^+ > \mathrm{r}^-$ (resp. $\uptau^+ > \mathrm{r}^+$)  one gets the position of an outlier $z > \varsigma_+$ by solving $P_{(\alpha_x,\alpha_y)}( \mathrm{r}^-,z^2) = 0$  (respectively $P_{(\alpha_x,\alpha_y)}( \mathrm{r}^+,z^2) = 0$), yielding the desired result and concluding the proof.   

\subsection{Proof of the Phase Transition for the Top Singular Vectors}
\label{sec:proof_singularvectors}

\begin{lem}
\label{lem:overlap_resolvent}
Under the same setting as in Thm.~\ref{thm:overlap}, we have:
\begin{align*}
     \langle \rdmvect{\Tilde{u}}^\pm_k, \vect{v}^{\star}_{x,k} \rangle^2 
     &=
     2  \lim_{z \to \Tilde{\sigma}^\pm_{k}} (z - \Tilde{\sigma}^\pm_{k}) \,  \Tilde{H}_{11,n}(z)
    \, ,
   \\
     \langle \rdmvect{\Tilde{v}}^\pm_k, \vect{v}^{\star}_{y,k} \rangle^2 
    &=
    2 \lim_{z \to \Tilde{\sigma}^\pm_{k}} (z - \Tilde{\sigma}^\pm_{k}) \,  \Tilde{H}_{33,n}(z) 
    \, , 
\end{align*}
where $\Tilde{H}_{ll,n}(z)$ denotes the $l$-th diagonal entry of the matrix $\matKspkk$ of Eq.~\eqref{eq:def_matK} and $\{\Tilde{\sigma}^{\pm}_{k}\}_{k=1}^r$ denotes the set of the top $2r$ singular values of $\matSspk$, indexed such that $\Tilde{\sigma}^{\pm}_{k} \to b(\mathrm{r}_k^\pm)$ as in Thm.~\ref{thm:singularvalue}. 
\end{lem}
\begin{proof}
    The proof is detailed in App.~\ref{sec:app:proof_lem_overlap}.
\end{proof}
\begin{lem}
\label{lem:concentration_diagonal4by4_spk}
For any $k \in \integset{r}$, the diagonal entries of the $(4 \times 4)$ matrix $\matKspkk(z)  = \big(\Tilde{H}_{ij,n}(z)\big)_{1 \leq i,j \leq 4}$ of the previous lemma are given asymptotically by 
\begin{align}
    \Tilde{H}_{ll,n} (z)\xrightarrow[n\to \infty]{\as} h_l(z) + \frac{f_{l,k}(z)}{j_k(z)} \quad \mbox{for } l \in \integset{4} ,
\end{align}
where $h_l$ are given in Lemma~\ref{lem:concentration_diagonal4by4}, $j_k$ is defined in Eq.~\eqref{eq:jk_ttransform} and $f_{l,k}$ is given in App.~\ref{sec:app:proof_lem_concentrations_spk}. In particular
\begin{align*}
    f_{1,k}(z) &= \frac{h_1(z) t(z^2)}{\alpha_x \lambda_y}\bigg( 1+\rho^2 \lambda_y -(1-\rho^2) \lambda_y \frac{t(z^2)}{\alpha_x}\bigg),  \\
f_{3,k}(z) &= \frac{h_3(z) t(z^2)}{\alpha_x \lambda_x}\bigg( 1+\rho^2 \lambda_x -(1-\rho^2) \lambda_x \frac{t(z^2)}{\alpha_x}\bigg),
\end{align*}
where we have dropped the index $k$ for $\lambda_{x,k},\lambda_{y,k},\rho_k$. 
\end{lem}
\begin{proof}
    See App.~\ref{sec:app:proof_lem_concentrations_spk}.
\end{proof}
Taking the limit $n\to \infty$ in Lemma~\ref{lem:overlap_resolvent} with the asymptotics of Lemma~\ref{lem:concentration_diagonal4by4_spk} one gets that if there is no outlier at $z$ then by Lemma~\ref{lem:condition_outlier} $j_k(z) \neq 0$ such that in this case, the associated overlap is asymptotically zero. Conversely, if there is an outlier at $b(\mathrm{r}_k^\pm)$, then we have  $j_k(b(\mathrm{r}_k^\pm)) =0$ and one needs to evaluate $\lim_{z \to b(\mathrm{r}_k^\pm)} 2 (z - b(\mathrm{r}_k^\pm)) f_{l,k}(z)/(j_k(z))$ which by  L'Hôpital's rule yields
\begin{align}
 \langle \rdmvect{\Tilde{u}}^\pm_k, \vect{v}^{\star}_{x,k} \rangle^2 
     &\xrightarrow[n \to \infty]{\as}    2 \frac{f_{1,k} \big( b(\mathrm{r}_k^\pm) \big)} { (j_k)'\big( b(\mathrm{r}_k^\pm)  \big) } \, ,
\end{align}
and similarly for $ \langle \rdmvect{\Tilde{v}}^\pm_k, \vect{v}^{\star}_{y,k} \rangle^2$ with $f_{1,k}$ replaced by $f_{3,k}$. From this point, one differentiates the function \( j_k(z) \) using its expression given in Eq.~\eqref{eq:jk_ttransform}, which yields
\begin{align}
    j'_k(z) 
    &=
    \frac{2 z t'(z^2)}{\alpha_x \lambda_{x,k} \lambda_{y,k}} R_{\lambda_{x,k},\lambda_{y,k},\rho_k}^\prime\bigg( \frac{t(z^2)}{\alpha_x} \bigg).
\end{align}
Differentiating Eq.~\eqref{eq:fixedpoint_transform} then gives 
\begin{align}
    t'(z^2) &= - \frac{t(z^2)^2}{\alpha_x Q_{(\alpha_x,\alpha_y)}(t(z^2)/\alpha_x)} ,
\end{align}
and substituting $z=b(\mathrm{r}^\pm)$ and $t(z^2) = \alpha_x \mathrm{r}^\pm$ then gives the desired result after simplification. 

\subsection{Acknowledgments}
The authors thank Yinuo Cheng for identifying the omission of a term in the limiting expression of  $\matKk$  presented in the first version of the paper.

\section{CONCLUSION}
\label{sec:conclusion}

In this work, we provided a rigorous analysis of the spectral properties of spiked cross-covariance matrices in the high-dimensional regime, when the signals across the two channels are partially correlated. Building on tools from random matrix theory, we show the emergence of a Baik Ben-Arous Péché (BBP)-type phase transition in the top singular values and quantified the alignment (overlap) of the corresponding singular vectors with the ground truth signals. As a consequence of these results, we obtain new theoretical insights into the behavior of Partial Least Squares (PLS) methods in high dimensions. In particular, we identified the conditions under which PLS can successfully recover signal structure and compare it with the single-channel setting. 

\bibliography{biblio}

\clearpage
\appendix
\thispagestyle{empty}

\onecolumn
\aistatstitle{
Supplementary Materials}
\section{Standard Linear Algebra and Random Matrix Theory Results}
\label{sec:app:linearalgebra}
In this section, we review some classical results from linear algebra and random matrix theory (RMT); see, e.g., \cite{benaych2012singular} for references.
\begin{lem}[Woodbury Resolvent identity]
\label{lem:woodbury-resolvent}
Let $\mat{A} \in \samatrix_n$, $\mat{Q} \in \matrixset_{n,d}$,  $\mat{C} \in \samatrix_d$,   $\mat{A'}= \mat{A} + \mat{Q}\mat{C}\mat{Q}^T$; $\mat{G}=(z\mat{I}-\mat{A})^{-1}$ and  $\mat{G'}=(z\mat{I}-\mat{A'})^{-1}$,then the following identity holds:
\begin{align}
   \mat{G'}
    &=
     \mat{G} +  \mat{G} \mat{Q} \mat{C}  ( \mat{I} - \mat{Q}^T \mat{G} \mat{Q} \mat{C} )^{-1}  \mat{Q}^T \mat{G} \, , 
\end{align}
in particular if we define $\mat{K}:=\mat{Q}^T \mat{G} \mat{Q}$ and $\mat{K'}:=\mat{Q}^T \mat{G'} \mat{Q}$, we have:
\begin{align}
\label{eq:woodbury_resolvent}
   \mat{K'}
    &=
     \mat{K} +   \mat{K} \mat{C}  ( \mat{I} -  \mat{K} \mat{C})^{-1}   \mat{K} \, .
\end{align}
\end{lem}

\begin{lem} 
\label{lem:overlap_from_resolvent}
Let $\mat{A} \in \samatrix_n$, $\mat{G}(z) := (z \mat{I}- \mat{A})^{-1}$, we denote by $\{ \vect{u}_i \}_{i \in \llbracket n \rrbracket}$ the unit-norm eigenvectors of $\mat{A}$ associated to the eigenvalues $\{\lambda_i\}_{i \in \llbracket n\rrbracket}$, counted with multiplicity. Fix $\lambda_{i_0} \in \mathrm{Spec}(\mat{A})$ and $\vect{q}\in \mathbb{R}^n$ then for any analytic function $f(\cdot)$ with $f(\lambda_{i_0}) \neq 0$,
we have 
\begin{align}
    \sum_{j: \lambda_j =\lambda_{i_0}} | \langle \vect{q},  \vect{u}_j   \rangle|^2 = \lim_{z \to \lambda_{i_0}} \frac{(z - \lambda_{i_0})}{f(\lambda_{i_0})} \cdot    \langle \vect{q}\, , f(\mat{A}) \mat{G}(z)   \vect{q} \rangle \, .
\end{align}
\end{lem}
\begin{proof}
    This follows directly from the eigendecomposition of $\mat{G}(z) = \frac{1}{z-\lambda_{i_0}} \sum_{i| \lambda_i=\lambda_{i_0}  } \vect{u}_i \vect{u}_i^T +  \sum_{i |  \lambda_i \neq \lambda_{i_0} } \frac{1}{z-\lambda_{i}}  \vect{u}_i \vect{u}_i^T $. 
\end{proof}

\begin{lem}[Eigenvalue decomposition]
\label{lem:decomposition-chiral}
    Let $\mat{B} \in \matrixset_{n,m}$ with singular value decomposition $\mat{B} = \mat{U} \mat{D} \mat{V}^T$, where $\mat{U} \in \matrixset_{n,t},\mat{V} \in \matrixset_{n,t} $ with orthonormal columns and $\mat{D}$ a $(r \times r)$ diagonal matrix. Introduce the three matrices $\mat{Q}= \begin{pmatrix}
        \mat{U} & \vect{0} \\
        \vect{0} & \mat{V}
    \end{pmatrix}$, $\mat{W}_\pm := \frac{1}{\sqrt{2}} \begin{pmatrix}
        \mat{U} \\
        \pm \mat{V}
    \end{pmatrix}$, then we have the following decompositions for the chiral extension of $\mat{B}$:
     \begin{align}
    (i)& \qquad  \girko(\mat{B}) = \mat{Q}
        \begin{pmatrix}
            \mat{0} & \mat{D} \\
            \mat{D} &  \mat{0}
        \end{pmatrix}
       \mat{Q}^T \, ; \\
    (ii)& \qquad \girko(\mat{B}) = [ \mat{W}_+ \, \,  \mat{W}_-  ]
        \begin{pmatrix}
            \mat{D} & \mat{0} \\
            \mat{0} & - \mat{D}
        \end{pmatrix}
        [ \mat{W}_+ \, \,  \mat{W}_-  ]^T \, .
    \end{align}
\end{lem}

\begin{lem}[Resolvent for Chiral Matrices]
\label{lem:block-decomposition-resolvent-chiral}
For $\mat{B} \in \nbymmatrix$, we have
\begin{align}
    \big( z \mat{I} - \girko(\mat{B}) \big)^{-1} &=
    \begin{pmatrix}
      z \cdot \mat{\Gamma}  &   \mat{\Gamma} \mat{B} \\
      \mat{B}^T \mat{\Gamma} & z \cdot \mat{\check{\Gamma}}
    \end{pmatrix}
\end{align}
with $\mat{\Gamma} := (z^2 \mat{I}- \mat{B}\mat{B}^T)^{-1}$ and $\mat{\check{\Gamma}} := (z^2 \mat{I}- \mat{B}^T\mat{B})^{-1} $. 
\end{lem}

\begin{lem}[matrix determinant lemma for chiral matrices]
\label{lem:det-lemma-chiral}
Let  $\mat{B} \in \nbymmatrix$, $\mat{L} \in \matrixset_{n,d}$, $\mat{C} \in \matrixset_{d,d}$ and invertible,  $\mat{R} \in \matrixset_{m,d}$, then for any $z \notin \mathrm{Spec} \big( \girko(\mat{B}) \big)$ we have:
\begin{align}
    \mathrm{det}(z \mat{I} - \girko( \mat{B} + \mat{L}\mat{C}\mat{R}^T))
    &=
       (-1)^d  \mathrm{det}( \mat{C})^2 \cdot     \mathrm{det}(z \mat{I} - \girko( \mat{B})) \cdot 
         \mathrm{det}(  \mat{\Upsilon}(z) ) \, ,
\end{align}
with the $(2 d \times 2 d)$ matrix $ \mat{\Upsilon}(z)$ given by
\begin{align}
    \mat{\Upsilon}(z) 
    :=  
    \begin{pmatrix}
        z \cdot \mat{L}^T \mat{\Gamma}  \mat{L} 
        &
        \mat{L}^T \mat{\Gamma}  \mat{B} \mat{R}
        \\
       (\mat{B} \mat{R})^T 
        \mat{\Gamma} \mat{L}
        &
        z \cdot \mat{R}^T 
       \mat{\check{\Gamma}} \mat{R}
    \end{pmatrix} 
    - \girko \big( (\mat{C}^{T})^{-1} \big) \, . 
\end{align}
and where $\mat{\Gamma} := (z^2 \mat{I}- \mat{B}\mat{B}^T)^{-1}$, $\mat{\check{\Gamma}} := (z^2 \mat{I}- \mat{B}^T\mat{B})^{-1} $. 
\end{lem}

\begin{lem}[Concentration Quadratic Form]

\label{lem:concentrationquadraticform}
Let $\mat{A} \in \mathbb{M}^{H}_{n, n} $ be  orthogonal-invariant in law with bounded operator norms. Suppose $\vect{u}_i, \vect{u}_j$ of unit norm, independent of $\mat{A}$ such that $\< \vect{u}_i , \vect{u}_j \> \to \varrho$ (with $\varrho=1$ for $i=j$) , then we have
    \begin{align}
        \bigg| \langle \vect{u}_i , \mat{A} \vect{u}_j \rangle - \varrho n^{-1} \Tr \mat{A}  \bigg| \xrightarrow[n\to\infty]{\as} 0 \, .
    \end{align}
\end{lem}

\begin{lem}
\label{lem:ap:identityproductmat}If we denote by $\mathrm{Spec}_+(\mat{A}) = \{ z \in \mathrm{Spec}(\mat{A}) \setminus 0 \}$, then for $\mat{A} \in \mathbb{R}^{n,m_1}$, $\mat{B} \in \mathbb{R}^{n,n}$, we have $\mathrm{Spec}_+(\mat{A}^T\mat{B}\mat{A}) =  \mathrm{Spec}_+( (\mat{A}\mat{A}^T)^{1/2} \mat{B}(\mat{A}\mat{A}^T)^{1/2})$. 
\end{lem}

For $\mu \in \mathcal{P}_c(\mathbb{R}_+)$ with no mass at the origin and rightmost edge $\lambda_+$, recall the definition of the $t$-transform that uniquely characterizes the distribution $\mu$:
\begin{align}
        t_{\mu}(z) := \int_{\mathbb{R}} (z-x)^{-1}x \, \mathrm{d}\mu(x) \, , 
\end{align}
which is strictly decreasing on $(\mathrm{\lambda}_+, \infty)$ and thus admits an inverse for the composition that we denote $t^{\langle -1 \rangle}$. The $S$-transform of a measure $\mu$ is defined as 
\begin{align}
    S_{\mu}(\theta) := \frac{\theta +1}{\theta t_{\mu}^{\langle -1 \rangle}(\theta)}
\end{align}
and appears naturally whenever one considers the product of two free elements:
\begin{lem}[Free multiplicative convolution]
\label{lem:freeconvolution}
Let $\mu, \nu \in \mathcal{P}_c(\mathbb{R}_+)$, there exists a unique distribution $\mu \boxtimes\nu$ known as the free convolution of $\mu$ and $\nu$ such that 
\begin{align}
    S_{\mu \boxtimes\nu}(\theta) = S_{\mu}(\theta) \, S_{\nu}(\theta) \, .
\end{align}
Furthermore for a sequence of $(n \times n)$ independent symmetric positive definite matrices $(\mat{A}_n, \mat{B}_n)_n$ such that $\mat{A}_n$ is orthogonally invariant in law ($\mat{A}_n \ed \mat{O}\mat{A}_n \mat{O}^T$ for any orthogonal matrix $\mat{O}$) and $\mu_{\mat{A}_n} \to \mu$ and $\mu_{\mat{B}_n} \to \nu$, we have
\begin{align}
    \mu_{\mat{A}_n^{1/2}\mat{B}_n\mat{A}_n^{1/2}} \to \mu \boxtimes \nu. 
\end{align}
\end{lem}

\section{Properties of the roots of the Edge and Secular Polynomials (Lemma~\ref{lem:roots_polynomials})}
\label{sec:app:cubicroots}

For $  Q_{(\alpha_x,\alpha_y)}$ defined by Eq.~\eqref{eq:def_Pol_Q}, we have $Q(0)=1$ and $\lim_{x\to \infty}   Q_{(\alpha_x,\alpha_y)}(x) = - \infty $, since $  Q_{(\alpha_x,\alpha_y)}'(x) = - 2(\alpha_x \alpha_y + \alpha_x + \alpha_y)x - 6\alpha_x \alpha_y x^2 $ is non-zero for $x>0$, the function admits exactly one positive root. 

Similarly for $R_{(\lambda_x,\lambda_y,\rho)}$ defined by Eq.~\eqref{eq:def_Pol_R}, the discriminant is given by 
\begin{align*}
 \mathrm{Dis}(  R_{(\lambda_x,\lambda_y,\rho)}) 
 &=  (\lambda_x - \lambda_y)^2 + 2 \rho^2 \lambda_x \lambda_y (\lambda_x + \lambda_y +2) + \rho^4 \lambda_x^2 \lambda_y^2 \ge 0 
\end{align*}
which is non-negative and vanishes only when both $\rho=0$ and $\lambda_x = \lambda_y$. For $ |\rho| \in (0,1)$ and $\lambda_x, \lambda_y >0$, the two non-negative roots are given by 
\begin{align}
\label{eq:explicit_rpm}
\mathrm{r}^{\pm}
&=
\frac{
\lambda_x+\lambda_y+\rho^2\lambda_x\lambda_y
\pm
\sqrt{
(\lambda_x-\lambda_y)^2
+\rho^2\lambda_x\lambda_y
\bigl(4+2\lambda_x+2\lambda_y+\rho^2\lambda_x\lambda_y\bigr)
}
}{
2(1-\rho^2)\lambda_x\lambda_y
}, 
\end{align}
with limiting values at $\rho\to 0$ and $|\rho|\to1$ given by
\begin{align}
\label{eq:behavior_root_rhoto0}
    \mathrm{r}^{-} \to \frac{1}{\max(\lambda_x,\lambda_y)} \quad \mbox{and} \quad \mathrm{r}^{+} \to \frac{1}{\min(\lambda_x,\lambda_y)} \quad \mbox{as } \rho\to0 ;
\end{align}
and
\begin{align}
\label{eq:behavior_root_rhoto1}
    \mathrm{r}^{-} \to \frac{1}{\lambda_x + \lambda_y + \lambda_x\lambda_y} \quad \mbox{and} \quad \mathrm{r}^{+} \to \infty \quad \mbox{as } |\rho|\to 1 .
\end{align}
Differentiating the expression $R(\mathrm{r}^\pm) =0$ with respect to $\rho^2$ gives (for a simple root):
\begin{align}
    \frac{\mathrm{d}\mathrm{r}^\pm}{\mathrm{d}\rho^2}
 &=
 \frac{\lambda_x \lambda_y\mathrm{r}^\pm(1+\mathrm{r}^\pm) }{R'(\mathrm{r}^\pm)}.
 \end{align}
Since $R'(\mathrm{r}^-) < 0 < R'(\mathrm{r}^+)$, the claimed monotonicity follows for each root.

\section{Bulk Distribution of Singular Values (Prop.~\ref{prop:bulk})}
\label{sec:app:bulk}
Since we assume $d_x \leq d_y$ without loss of generality, the singular values of $\matS$ are by definition the square root of the eigenvalues of the matrix $\matX^T\matY \matY^T \matX$ and by  Lemma~\ref{lem:ap:identityproductmat} one has
\begin{align}
    \mathrm{Spec}_+( \matX^T\matY \matY^T \matX) = \mathrm{Spec}_+ \big( (\matX \matX^T)^{1/2} \matY \matY^T (\matX \matX^T)^{1/2} \big) \, ,
\end{align}
from which one can check that the $T$-transform of the two measures are related by 
\begin{align}\label{eq:ttransf_freeproduct}
    t_{(\matX \matX^T)^{1/2} \matY \matY^T (\matX \matX^T)^{1/2}}(z) = \frac{1}{\alpha_{x}} t_{\matS \matS^T}(z) \qquad \mbox{with} \; \alpha_{x} = \lim_{n\to\infty} n/d_x(n) \, . 
\end{align}
From classical random matrix result \cite{potters2020first}, the limiting distribution of  $\matX\matX^T$ (respectively of  $\matY\matY^T$) is known  to be given by the Marchenko-Pastur distribution of aspect ratio $\alpha_x$ (resp. $\alpha_y$) whose $S$-transform is 
\begin{align}
    S_{\mu_{\mathrm{MP}(\alpha)}}(\theta) = (1 + \alpha \theta)^{-1} \,, 
\end{align}
with $\alpha = \alpha_x$ (resp. $\alpha_y$). 
\vspace{0.5em}
Next by Assumption (\textbf{A1}), the two  matrices $\matX, \matY$ are independent with Gaussian entries, hence the matrices $\matX\matX^T$ is orthogonally invariant and thus using Lemma~\ref{lem:freeconvolution}, the limiting spectral distribution of $ (\matX \matX^T)^{1/2} \matY \matY^T (\matX \matX^T)^{1/2}$ is given by the free convolution $\mu_{\mathrm{MP}(\alpha_x)} \boxtimes \mu_{\mathrm{MP}(\alpha_y)}$ whose $T$-transform $t_{\boxtimes} \equiv t_{\mu_{\mathrm{MP}(\alpha_x)} \boxtimes \mu_{\mathrm{MP}(\alpha_y)}}$  satisfies: 
\begin{align}
t_{\boxtimes}(z) z &= (1+t_{\boxtimes}(z))  ( 1 + \alpha_x t_{\boxtimes}(z)) (1 + \alpha_y t_{\boxtimes}(z)) \, , 
\end{align}
and by Eq.~\eqref{eq:ttransf_freeproduct} the one for the distribution of the square of the singular values of $\matS$ is giving simply by $t(z) = \alpha_x t_{\boxtimes}(z)$ from which one reads the desired result. 

\section{Determinant Equation for the eigenvalues (Lemma ~\ref{lem:condition_outlier})}
\label{sec:app:proof_lem_condition}

If we denote by $\matS := \matX^T \matY$, the cross-covariance matrix without any spiked, we have the decomposition
\begin{align}
    \matSspk &=  \matS + \sum_{k=1}^r \rdmmat{P}_k + \mathbf{E}_n \, ,
\end{align}
where  $\mathbf{E}_n = \sum_{k \neq k'} \sqrt{\lambda_{x,k}, \lambda_{y,k'}} \langle \usxk , \mathbf{u}^\star_{y,k'}  \rangle \vsxk \mathbf{v}_{s,k'}^T$ satisfies $\|\mathbf{E}_n \|_{\mathrm{op}} \to 0$ by Assumption (A2) and the matrices $\rdmmat{P}_k$ are given by
\begin{align}
    \rdmmat{P}_k &= 
    \sqrt{ \lambda_{x,k} \lambda_{y,k} } \, \langle \usxk, \usyk \rangle \cdot  \vsxk  (\vsyk)^{T} 
    + \sqrt{ \lambda_{y,k} } \cdot ( \matX^{T} \usyk )  (\vsyk)^T \\
    & \quad + \sqrt{ \lambda_{x,k} } \cdot \vsxk \big( \matY^{T} \usxk \big)^T \, , 
\end{align}
such that the matrix $\rdmmat{P}= \sum_{k=1}^r \rdmmat{P}_k$ can be rewritten as
\begin{align}
    \rdmmat{P} &= 
    ( \vect{v}^\star_{x,1} , \matX^{T} \vect{u}^\star_{y,1} , \dots, \vect{v}^\star_{x,r} , \matX^{T}\vect{u}^\star_{y,r} ) \\
    &\quad \quad \mathrm{BlockDiag}( \mat{B}_1, \dots, \mat{B}_r)
    (  \vect{v}^\star_{y,1} , \matY^{T} \vect{u}^\star_{x,1}, \dots,  \vect{v}^\star_{y,r} ,  \matY^{T}\vect{u}^\star_{x,r} )^T \, ,
\end{align}
with the $(2 \times 2)$ matrices $\mat{B}_k$ defined by
\begin{align}
    \mat{B}_k:= 
    \begin{pmatrix}
        \sqrt{\lambda_{x,k} \lambda_{y,k}} \, \rho_{k}  & \sqrt{\lambda_{x,k}} \\
        \sqrt{\lambda_{y,k}} & 0
    \end{pmatrix} \qquad \mbox{with } \rho_{k} := \langle \usxk, \usyk \rangle \, .
\end{align}
For $\lambda_{x,k},\lambda_{y,k} >0$, the matrix $\mat{B}_k$ is invertible with inverse
\begin{align}
\label{eq:inversematA}
    \mat{B}_k^{-1}=
    \begin{pmatrix}
    0 & \lambda_{y,k}^{-1/2} \\
    \lambda_{x,k}^{-1/2} & - \rho_{k}
    \end{pmatrix} \, .
\end{align}
Next, to ease notations we denote by
\begin{align}
    \rdmmat{\Tilde{Z}}_{n}  &= \eta(\matSspk - \mat{E}_n ) := 
    \begin{pmatrix}
       \mat{0} & \matSspk - \mat{E}_n \\
    (\matSspk - \mat{E}_n)^T & \mat{0} 
    \end{pmatrix} 
    \qquad \mbox{and} \qquad 
     \rdmmat{Z}_{n}  = \eta(\matS) := 
    \begin{pmatrix}
       \mat{0} & \matS \\
    \matS^T & \mat{0} 
    \end{pmatrix} \, ,
\end{align}
by determinant lemma \ref{lem:det-lemma-chiral}, the characteristic polynomial of $\rdmmat{\Tilde{Z}}_{n} $ is given by 
\begin{align}
    \det \big( z \mat{I} -  \rdmmat{\Tilde{Z}}_{n}  \big) 
    =
    \prod_{k=1}^r (\det \mat{B}_k)^2 
    \cdot 
    \det \big( z \mat{I} -  \rdmmat{Z}_n  \big) \cdot \det \rdmmat{M}_n(z) 
\end{align}
with the matrix $\rdmmat{M}_n(z)$
\begin{align}
    \rdmmat{M}_n(z) &:=    \mathrm{BlockDiag} \left( \begin{pmatrix}
        \mat{0} & \mat{B}_k \\
        \mat{B}_k^T & \mat{0}
    \end{pmatrix}_{k=1,\dots,r} \right)^{-1}  - \,  \rdmmat{\Theta}_n(z) 
\end{align}
where $\bm\Theta_n(z)=\bm Q^T\bm G(z)\bm Q$ with $\bm Q=[,\bm Q_1^\star\ \cdots\ \bm Q_r^\star,]$ and 
 since the components satisfy Assumption \textbf{A2},  we have by Lemma~\ref{lem:concentrationquadraticform}  the elements outside the block-diagonal of $ \rdmmat{\Theta}_n(z) $vanishes exponentially fast  and $\rho_{n,k} \to \rho_k$, leading to
\begin{align}
   \| \rdmmat{\Theta}_n(z) - \mathrm{BlockDiag} (\matKk) \|_{\mathrm{op}} \to 0 \ , 
\end{align}
with $\matKk$ defined in Eq.~\eqref{eq:def_matK}. Since the determinant of block-diagonal matrix is the product of the determinant of its blocks and the inverse of $\mat{B}_k$ is given by Eq.~\eqref{eq:inversematA}, one gets the desired result.

\section{ Concentration of the diagonal entries (Lemma~\ref{lem:concentration_diagonal4by4})}
\label{sec:app:proof_lem_concentration}

We fix an index $k$ in the rest of this section and  recall  $H_{ij,n}(z)=(\bm H_{k,n}(z))_{ij}$. We first show that each diagonal entry can be approximated by the trace of certain random matrices. To ease notation, in the following paragraph we write $\vsx \equiv \vsxk$ (and similarly for other components) as the computation is the same for each block. 
\begin{lem}[Diagonal Entries] The diagonal entries of $\matKk$ satisfy
\begin{align} 
    \Big| h_{1,n}(z)  - \frac{z}{d_x}  \mathbb{E}\mathrm{Tr} (z^2 - \matS \matS^T)^{-1}  \Big| &\to 0 
    \, , \\
    \Big|h_{2,n}(z)  - \frac{z}{n} \mathbb{E}\mathrm{Tr}   \rdmmat{X}(z^2 - \matS \matS^T)^{-1} \rdmmat{X}^T  \Big| &\to 0
    \, , \\
    \Big| h_{3,n}(z)- \frac{z}{d_y} \mathbb{E}\mathrm{Tr}  (z^2 - \matS^T\matS)^{-1} \Big| &\to 0
    \, , \\
   \Big| h_{4,n}(z)- \frac{z}{n} \mathbb{E}\mathrm{Tr}  \rdmmat{Y}(z^2 - \matS^T\matS)^{-1} \rdmmat{Y}^{T} \Big|
    &\to 0 
    \, . 
\end{align}
\end{lem}

\begin{proof}
    By definition of the matrix $\rdmmat{H}$, its diagonal entries $\{H_{ii,n}(z)\}_{i \in \llbracket 4 \rrbracket}$ are given by
\begin{align} 
   H_{11,n}(z) &:= \langle (\vsx \quad \vect{0}),  \rdmmat{G}(z) (\vsx \quad \vect{0}) \rangle \, , \\
     H_{22,n}(z) &:= \langle (\rdmmat{X}^T \usy \quad \vect{0}),  \rdmmat{G}(z)(\rdmmat{X}^T \usy  \quad \vect{0}) \rangle  \, , \\
     H_{33,n}(z) &:= \langle (\vect{0} \quad \vsy),  \rdmmat{G}(z)(\vect{0} \quad \vsy) \rangle  \, , \\
     H_{44,n}(z) &:= \langle (\vect{0} \quad \rdmmat{Y}^T\usx  ),  \rdmmat{G}(z)(\vect{0} \quad \rdmmat{Y}^{T}\usx  ) \rangle  \, , 
\end{align}
and thus correspond to quadratic forms with either the top-left corner (for $H_{11,n}, H_{22,n}$) or bottom-right corner (for $H_{33,n}, H_{44,n}$) of the resolvent matrix $\rdmmat{G}(z)$ of the Hermitian matrix $\girko(\matS)$. By properties of resolvent of chiral matrices (see Lemma~\ref{lem:block-decomposition-resolvent-chiral}) the latter projections are given as the resolvent of $\matS \matS^T$ (resp. of $\matS^T \matS$) evaluated at $z^2$, such that we have:
\begin{align} 
    H_{11,n}(z) &= \langle \vsx ,  z(z^2 - \matS \matS^T)^{-1} \vsx  \rangle \, , \\
    H_{22,n}(z) &= \langle \usy ,   z\rdmmat{X}(z^2 - \matS \matS^T)^{-1} \rdmmat{X}^T \usy \rangle  \, , \\
    H_{33,n}(z) &= \langle \vsy,  z(z^2 - \matS^T\matS)^{-1} \vsy\rangle  \, , \\
    H_{44,n}(z) &= \langle \usx  , \rdmmat{Y} z(z^2 - \matS^T\matS)^{-1} \rdmmat{Y}^{T} \usx  \rangle  \, , 
\end{align}
the limits of these quadratic forms are captured by Lemma~\ref{lem:concentrationquadraticform} 
\end{proof}

\begin{lem}[Off Diagonal Entries] For $z^2 \in \mathbb{C} \setminus \mathrm{Supp} \mu_{(\alpha_x,\alpha_y)}$, the off-diagonal entry $H_{24,n} = (\matKk)_{24}$ and $H_{42,n} = (\matKk)_{42} = H_{24,n}$ by symmetry, satisfy 
\begin{align}
    \Big| H_{24,n}(z)  - \frac{\rho_k}{n}  \mathrm{Tr} (z^2 - \matS \matS^T)^{-1} \matS \matS^T \Big| &\to 0 
    \, ,
\end{align}
and all other off-diagonal terms vanish.
\end{lem}
\begin{proof}
We have
\begin{align}
    H_{24,n} (z) 
    &=
  \left\langle
(\rdmmat X^T\usyk \quad \vect 0),
\rdmmat G(z)
(\vect 0 \quad \rdmmat Y^T\usxk)
\right\rangle
\end{align}
such that by Lemma~\ref{lem:block-decomposition-resolvent-chiral} we have
\begin{align}
    H_{24,n} (z) 
    &=
    \langle \usyk , \bm{X} (z^2 - \matS \matS^T)^{-1}\matS\bm{Y}^T \usxk \rangle
\end{align}
and since $\langle \usxk , \usyk \rangle \to \rho_k$, the desired result follows by applying Lemma~\ref{lem:concentrationquadraticform} and using the cyclicity of the trace
\begin{align}
    \mathrm{Tr}  \bm{X} (z^2 - \matS \matS^T)^{-1}\matS\bm{Y}^T
    &=
    \mathrm{Tr}   (z^2 - \matS \matS^T)^{-1}\matS\bm{Y}^T\bm{X} =   \mathrm{Tr}   (z^2 - \matS \matS^T)^{-1}\matS\matS^T.
\end{align}
The remaining off-diagonal entries converge to zero as they include quadratic forms with directions $\vsxk$ and $\vsyk$ paired with noise.  
\end{proof}

Taking the limit $n \to \infty$, one gets immediately the desired result for Part-(i) of the Lemma. For Part-(ii), we first use the intermediate lemma: 
\begin{lem} Set $\rdmmat{\Gamma}_n = (z^2 - \matS \matS^T)^{-1} $ and $\rdmmat{\check{\Gamma}}_n = (z^2 - \matS^T\matS)^{-1}$, then we have
    \begin{align}
        \frac{z}{n} \Tr \rdmmat{X}\rdmmat{\Gamma}_n\rdmmat{X}^{T}  
        &=
        \frac{z}{n} \Tr \, \mat{G}_{\matSx^{1/2} \matSy  \matSx^{1/2}}(z^2)  \matSx 
        \\
        \frac{z}{n} \Tr \rdmmat{Y}\rdmmat{\check{\Gamma}}_n\rdmmat{Y}^{T}  &=
        \frac{z}{n} \Tr \, \mat{G}_{\matSy^{1/2} \matSx  \matSy^{1/2}}(z^2)  \matSy 
    \end{align}
with $\mat{G}_{\mat{A}}(z) := (z -\mat{A})^{-1}$ the resolvent and $\rdmmat{S}_x := \matX\matX^T$,  $\rdmmat{S}_y := \matY\matY^T$.   
\end{lem}
\begin{proof} Taking $z$ sufficiently away, one has 
    \begin{align}
     \frac{z}{n} \Tr \rdmmat{X}\rdmmat{\Gamma}_n\rdmmat{X}^{T} 
     &= 
    \frac{z}{n} \Tr \rdmmat{X} \left( z^{-2} \cdot \sum_{k=0}^{\infty} z^{-2k} \Big(\rdmmat{X}^T \rdmmat{Y} \rdmmat{Y}^T \rdmmat{X}\Big)^k  \right) \rdmmat{X}^{T} \, , \\
    &=
    \frac{z}{n} \Tr  \left( z^{-2} \cdot \sum_{k=0}^{\infty} z^{-2k} \Big( \matSx \matSy \Big)^k  \right) \matSx \, , \\
    &=
    \frac{z}{n} \Tr  \left( z^{-2} \cdot \sum_{k=0}^{\infty} z^{-2k} \matSx^{1/2} \Big( \matSx^{1/2} \matSy  \matSx^{1/2} \Big)^k  \right) \matSx^{1/2} \, , \\
     &=
    \frac{z}{n} \Tr  \left( z^{-2} \cdot \sum_{k=0}^{\infty} z^{-2k} \Big( \matSx^{1/2} \matSy  \matSx^{1/2} \Big)^k  \right) \matSx \, , \\
     &=
    \frac{z}{n} \Tr  \left( z^2 -  \matSx^{1/2} \matSy  \matSx^{1/2} \right)^{-1} \matSx \, , \\
    &= 
    \frac{z}{n} \Tr \, \mat{G}_{\matSx^{1/2} \matSy  \matSx^{1/2}}(z^2)  \matSx \, , 
\end{align}
and the proof of the other part follows identically.
\end{proof}
\vspace{0.5em}
Next, we use the following property, which follows from subordination relation in free probability:
\begin{lem} Let $t_{\boxtimes}$ be the $T$-transform of the measure defined in Lemma~\ref{lem:freeconvolution}, then we have
    \begin{align}
        \frac{1}{n} \mathbb{E}\Tr \, \mat{G}_{\matSx^{1/2} \matSy  \matSx^{1/2}}(z^2)  \matSx 
        &\to
        \frac{ t_{\boxtimes}(z^2)}{1+ \alpha_y t_{\boxtimes}(z^2)}
        \\
        \frac{1}{n} \mathbb{E} \Tr \, \mat{G}_{\matSy^{1/2} \matSx  \matSy^{1/2}}(z^2)  \matSy 
        &\to
       \frac{t_{\boxtimes}(z^2)}{1+ \alpha_x t_{\boxtimes}(z^2)}
    \end{align}
\end{lem}
\begin{proof}
    See for example \cite{belinschi2007new}. A similar derivation can be found in Chapter 19 of \cite{potters2020first}.
\end{proof}
\vspace{0.5em}
To conclude, one uses the limit $t = t_{\boxtimes} \alpha_x$ of Eq.~\eqref{eq:ttransf_freeproduct} to express everything in term of the $T$-transform of the squared singular values and obtains the desired result.

\section{Proof of the identity between overlaps and resolvent (Lemma ~\ref{lem:overlap_resolvent})}
\label{sec:app:proof_lem_overlap}
This follows simply by applying the result of Lemma~\ref{lem:overlap_from_resolvent} that relates the overlap of the symmetric matrix $\eta(\matSspk)$ as a proper limit of the quadratic form of the resolvent with the associated vector and then using Lemma~\ref{lem:decomposition-chiral} to relate the eigenvectors of $\eta(\matSspk)$ to the singular vectors of $\matSspk$. 

\section{Proof of Lemma~\ref{lem:concentration_diagonal4by4_spk}}
\label{sec:app:proof_lem_concentrations_spk}

We fix an index $k$ and work within each block. Let
\begin{align}
\label{eq:defmatKk_app}
    \rdmmat{H}_k
    &:=
    \begin{pmatrix}
        h_1 &0 &0 &0 \\
        0   &h_2 &0 & \rho_k t(z^2)/\alpha_x \\
        0 & 0 & h_3 & 0 \\
        0 & \rho_k t(z^2)/\alpha_x  &0 & h_4 
    \end{pmatrix}, 
    \quad \mat{C}_k = \mat{A}_k^{-1},  \quad \mat{J}_k = \mat{H}_k - \mat{A}_k
\end{align}
by Woodbury identity Eq.~\eqref{eq:woodbury_resolvent}, we can decompose the matrix $\matKspkk$ of Eq.~\eqref{eq:def_matK}  as
\begin{align}
\label{eq:overlapmatrixdecomposition}
    \matKspkk 
    &= 
     \rdmmat{H}_k +  \rdmmat{H}_k\mat{C}_k( \bm{I}_k -  \rdmmat{H}_k \mat{C}_k)^{-1}  \rdmmat{H}_k + \bm{E}_n,
     \\
 \label{eq:overlapmatrixdecomposition2}    
    \matKspkk
    &=
     \rdmmat{H}_k -  \rdmmat{H}_k \mat{J}_k^{-1}  \rdmmat{H}_k + \bm{E}_n. 
\end{align}
since $\bm{I}_k -  \rdmmat{H}_k \mat{C}_k = - \mat{J}_k \mat{C}_k$, with $\| \bm{E}_n \|_{\mathrm{op}} \to 0$ locally uniformly away from the support of the limiting spectral distribution. Explicitly, we have
    \begin{align}
\mat J_k&=
\begin{pmatrix}
h_1&0&0&-\lambda_{x,k}^{-1/2}\\
0&h_2&-\lambda_{y,k}^{-1/2}&h_5\\
0&-\lambda_{y,k}^{-1/2}&h_3&0\\
-\lambda_{x,k}^{-1/2}&h_5&0&h_4
\end{pmatrix} \qquad \mbox{with } h_5=\rho_k\bigg(1+\frac{t(z^2)}{\alpha_x}\bigg), 
\end{align}
and $h_1,\dots,h_4$ given in Lemma~\ref{lem:concentration_diagonal4by4}. Its determinant is given by
\begin{align}
    j_k(z) 
    &= 
    \det \mat J_k 
    \, \\
    j_k(z) 
    &= 
    (h_1h_4-\lambda_{x,k}^{-1})(h_2h_3-\lambda_{y,k}^{-1})-h_1h_3h_5^2
    \, , \\
     j_k(z)
     &=
\frac{R_{(\lambda_{x,k},\lambda_{y,k},\rho_k)}(t(z^2)/\alpha_x)}{\lambda_{x,k}\lambda_{y,k}},
\end{align}
where we have used that $h_1 h_4 = h_2 h_3 = t(z^2)/\alpha_x$ and $h_1 h_3 = t(z^2)/(\alpha_x + t(z^2))$. Its inverse is given by the inverse of $(4 \times 4)$ matrices which writes in this case
\begin{align}
\mat J_k^{-1}=\frac1{j_k}
\begin{pmatrix}
h_4(h_2h_3-\lambda_{y,k}^{-1})-h_3h_5^2
&-h_3h_5/\sqrt{\lambda_{x,k}}
&-h_5/\sqrt{\lambda_{x,k}\lambda_{y,k}}
&(h_2h_3-\lambda_{y,k}^{-1})/\sqrt{\lambda_{x,k}}\\
-h_3h_5/\sqrt{\lambda_{x,k}}
&h_3(h_1h_4-\lambda_{x,k}^{-1})
&(h_1h_4-\lambda_{x,k}^{-1})/\sqrt{\lambda_{y,k}}
&-h_1h_3h_5\\
-h_5/\sqrt{\lambda_{x,k}\lambda_{y,k}}
&(h_1h_4-\lambda_{x,k}^{-1})/\sqrt{\lambda_{y,k}}
&h_2(h_1h_4-\lambda_{x,k}^{-1})-h_1h_5^2
&-h_1h_5/\sqrt{\lambda_{y,k}}\\
(h_2h_3-\lambda_{y,k}^{-1})/\sqrt{\lambda_{x,k}}
&-h_1h_3h_5
&-h_1h_5/\sqrt{\lambda_{y,k}}
&h_1(h_2h_3-\lambda_{y,k}^{-1})
\end{pmatrix}.
\end{align}
Multiplying by $ \rdmmat{H}_k$ given in Eq.~\eqref{eq:defmatKk_app} on the left and the right, we can express the diagonal entries of $\matKspkk$ thanks to  Eq.~\eqref{eq:overlapmatrixdecomposition2} explicitly:
\begin{align}
\Tilde H_{ll,n}(z)&=h_l(z)+\frac{f_{l,k}(z)}{j_k(z)}+\epsilon_{l,n}(z),\qquad l\in\integset{4},
\end{align}
where $\epsilon_{l,n}\to0$  and
\begin{align}
f_{1,k}&=-h_1^2\big(h_4(h_2h_3-\lambda_{y,k}^{-1})-h_3h_5^2\big),\\
f_{2,k}&=-h_2^2h_3(h_1h_4-\lambda_{x,k}^{-1})
+2\rho_k\frac{t(z^2)}{\alpha_x}h_1h_2h_3h_5
-\rho_k^2\frac{t(z^2)^2}{\alpha_x^2}h_1(h_2h_3-\lambda_{y,k}^{-1}),\\
f_{3,k}&=-h_3^2\big(h_2(h_1h_4-\lambda_{x,k}^{-1})-h_1h_5^2\big),\\
f_{4,k}&=-\rho_k^2\frac{t(z^2)^2}{\alpha_x^2}h_3(h_1h_4-\lambda_{x,k}^{-1})
+2\rho_k\frac{t(z^2)}{\alpha_x}h_1h_3h_4h_5
-h_4^2h_1(h_2h_3-\lambda_{y,k}^{-1}).
\end{align}
For the first entry, the cofactor terms cancel according to
\begin{align}
j_k\bigg(h_1+\frac{f_{1,k}}{j_k}\bigg)
&=-\frac{h_1}{\lambda_{x,k}}(h_2h_3-\lambda_{y,k}^{-1})
=\frac{h_1(1-\lambda_{y,k}t(z^2)/\alpha_x)}{\lambda_{x,k}\lambda_{y,k}}.
\end{align}
The same calculation for the third entry therefore gives
\begin{align}
\Tilde H_{11,n}(z)&=
\frac{h_1(z)(1-\lambda_{y,k}t(z^2)/\alpha_x)}{R_{(\lambda_{x,k},\lambda_{y,k},\rho_k)}(t(z^2)/\alpha_x)}+\epsilon_{1,n}(z),\\
\Tilde H_{33,n}(z)&=
\frac{h_3(z)(1-\lambda_{x,k}t(z^2)/\alpha_x)}{R_{(\lambda_{x,k},\lambda_{y,k},\rho_k)}(t(z^2)/\alpha_x)}+\epsilon_{3,n}(z).
\end{align}
Finally, subtracting $h_1$ and $h_3$ from these expressions yields
\begin{align}
f_{1,k}(z)&=\frac{h_1(z)t(z^2)}{\alpha_x\lambda_{y,k}}
\bigg(1+\rho_k^2\lambda_{y,k}-(1-\rho_k^2)\lambda_{y,k}\frac{t(z^2)}{\alpha_x}\bigg),\\
f_{3,k}(z)&=\frac{h_3(z)t(z^2)}{\alpha_x\lambda_{x,k}}
\bigg(1+\rho_k^2\lambda_{x,k}-(1-\rho_k^2)\lambda_{x,k}\frac{t(z^2)}{\alpha_x}\bigg),
\end{align}
which are the expressions used in Lemma~\ref{lem:concentration_diagonal4by4_spk}.

Similarly, using $h_2=h_1(1+t(z^2)/\alpha_x)$ and $h_4=h_3(1+t(z^2)/\alpha_x)$, the second and fourth entries simplify to
\begin{align}
f_{2,k}(z)
&=h_1(z)\bigg[
-\frac{t(z^2)}{\alpha_x}\bigg(1+\frac{t(z^2)}{\alpha_x}\bigg)
\bigg(\frac{t(z^2)}{\alpha_x}-\frac1{\lambda_{x,k}}\bigg)
+\rho_k^2\frac{t(z^2)^2}{\alpha_x^2}
\bigg(2+\frac{t(z^2)}{\alpha_x}+\frac1{\lambda_{y,k}}\bigg)
\bigg],\\
f_{4,k}(z)
&=h_3(z)\bigg[
-\frac{t(z^2)}{\alpha_x}\bigg(1+\frac{t(z^2)}{\alpha_x}\bigg)
\bigg(\frac{t(z^2)}{\alpha_x}-\frac1{\lambda_{y,k}}\bigg)
+\rho_k^2\frac{t(z^2)^2}{\alpha_x^2}
\bigg(2+\frac{t(z^2)}{\alpha_x}+\frac1{\lambda_{x,k}}\bigg)
\bigg].
\label{eq:f24_simplified}
\end{align}
For completeness, adding $h_2$ and $h_4$ gives
\begin{align}
\Tilde H_{22,n}(z)
&=\frac{h_1(z)}{R_{(\lambda_{x,k},\lambda_{y,k},\rho_k)}(t(z^2)/\alpha_x)}
\bigg[\bigg(1+\frac{t(z^2)}{\alpha_x}\bigg)
\bigg(1-\lambda_{x,k}\frac{t(z^2)}{\alpha_x}\bigg)
+\rho_k^2\lambda_{x,k}\frac{t(z^2)}{\alpha_x}
\bigg(\frac{t(z^2)}{\alpha_x}-\lambda_{y,k}\bigg)\bigg]
+\epsilon_{2,n}(z),\\
\Tilde H_{44,n}(z)
&=\frac{h_3(z)}{R_{(\lambda_{x,k},\lambda_{y,k},\rho_k)}(t(z^2)/\alpha_x)}
\bigg[\bigg(1+\frac{t(z^2)}{\alpha_x}\bigg)
\bigg(1-\lambda_{y,k}\frac{t(z^2)}{\alpha_x}\bigg)
+\rho_k^2\lambda_{y,k}\frac{t(z^2)}{\alpha_x}
\bigg(\frac{t(z^2)}{\alpha_x}-\lambda_{x,k}\bigg)\bigg]
+\epsilon_{4,n}(z),
\label{eq:H24_simplified}
\end{align}
where $\epsilon_{2,n},\epsilon_{4,n}\to0$. These expressions will be used to compute the overlaps of the scores.

\section{Proof of the detection comparison between PLS and CCA (Prop.~\ref{prop:PLS_vs_CCA})} 
\label{sec:PLS_VS_CCA}
We  again write $Q \equiv Q_{(\alpha_x,\alpha_y)}$ and $R \equiv R_{(\lambda_{x,k},\lambda_{y,k},\rho_k)}$ and omit the index $k$ in the rest of this section.

We first establish an inequality for the PLS edge parameter $\uptau^+$.
Since $\alpha_x,\alpha_y>1$ and
$\alpha_x\alpha_y\leq(\alpha_x+\alpha_y)^2/4$, we have
\begin{align}
Q\bigg(\frac1{\alpha_x+\alpha_y}\bigg)
&=1-\frac1{\alpha_x+\alpha_y}
-\frac{\alpha_x\alpha_y}{(\alpha_x+\alpha_y)^2}
-\frac{2\alpha_x\alpha_y}{(\alpha_x+\alpha_y)^3}
\geq\frac34-\frac{3}{2(\alpha_x+\alpha_y)}>0.
\end{align}
Because $Q$ is strictly decreasing on $(0,\infty)$, this implies
$(\alpha_x+\alpha_y)\uptau^+>1$.
Using $Q(\uptau^+)=0$ gives the identity
\begin{align}
\big(\alpha_x\alpha_y\uptau^+(1+\uptau^+)\big)^2
-(\alpha_x-1)(\alpha_y-1)
&=\big((\alpha_x+\alpha_y)\uptau^+
+\alpha_x\alpha_y(\uptau^+)^2\big)^2-1>0.
\end{align}
Since all quantities are positive,
\begin{align}
\alpha_x\alpha_y\uptau^+(1+\uptau^+)
>\sqrt{(\alpha_x-1)(\alpha_y-1)}.
\end{align}

Suppose now that PLS has no separated outlier, that is
$\mathrm r^-\geq\uptau^+$.
The smaller root satisfies
$\mathrm r^-\leq\min(1/\lambda_x,1/\lambda_y)$: indeed, $R(0)=1$ and evaluating $R$ at the minimum of $1/\lambda_x$ and $1/\lambda_y$ gives a non-positive value.
Writing the secular equation in factored form gives
\begin{align}
(1-\lambda_x\mathrm r^-)(1-\lambda_y\mathrm r^-)
&=\rho^2\lambda_x\lambda_y\mathrm r^-(1+\mathrm r^-).
\end{align}
The two factors on the left belong to $[0,1]$, hence
\begin{align}
\frac{\rho^2\lambda_x\lambda_y}
{(\alpha_x+\lambda_x)(\alpha_y+\lambda_y)}
&=\frac{(1-\lambda_x\mathrm r^-)(1-\lambda_y\mathrm r^-)}
{\mathrm r^-(1+\mathrm r^-)(\alpha_x+\lambda_x)(\alpha_y+\lambda_y)}\\
&\leq\frac1{\alpha_x\alpha_y\mathrm r^-(1+\mathrm r^-)}\\
&\leq\frac1{\alpha_x\alpha_y\uptau^+(1+\uptau^+)}\\
&<\frac1{\sqrt{(\alpha_x-1)(\alpha_y-1)}}.
\end{align}
Thus CCA has no separated outlier. Contraposition proves the implication. This phenomenon is illustrated in Fig.~\ref{add_fig:PLS_VS_CCA}.

\begin{figure*}[h]
    \centering
    \begin{subfigure}[t]{0.32\textwidth}
        \centering
\includegraphics[width=\textwidth]{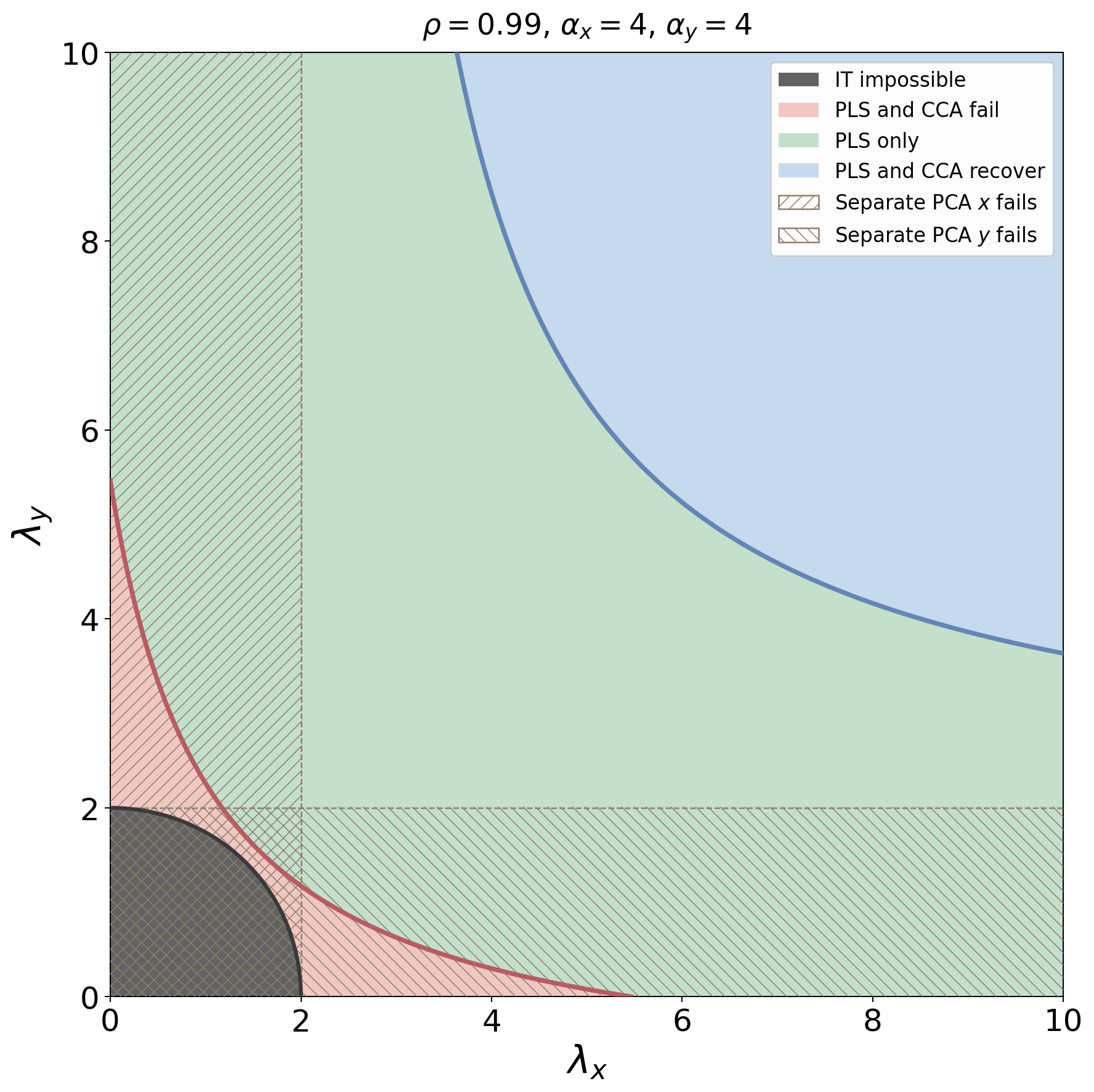}
    \end{subfigure}
    \hfill
    \begin{subfigure}[t]{0.32\textwidth}
        \centering
        \includegraphics[width=\textwidth]{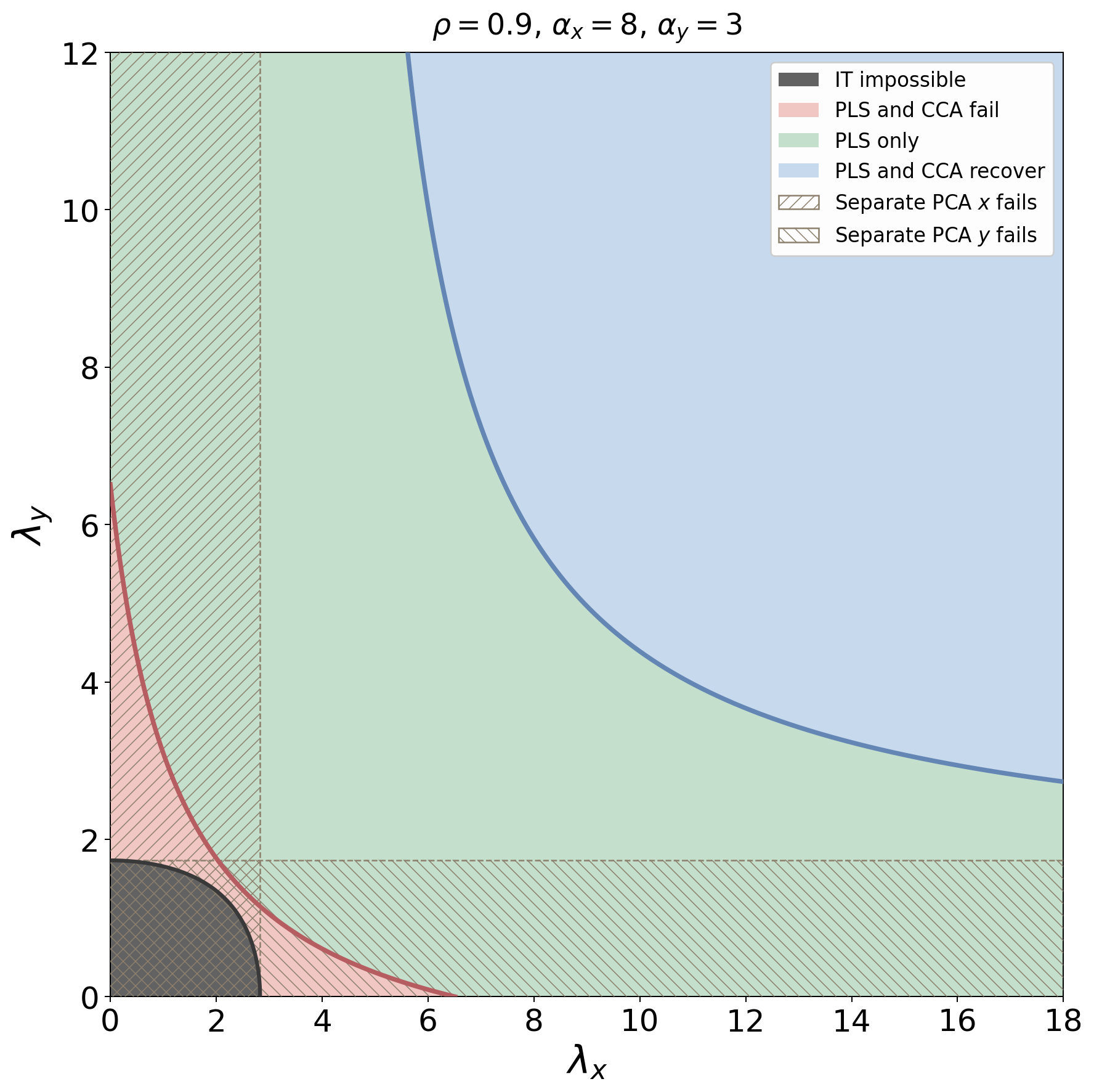}
    \end{subfigure}
    \hfill
    \begin{subfigure}[t]{0.32\textwidth}
        \centering
        \includegraphics[width=\textwidth]{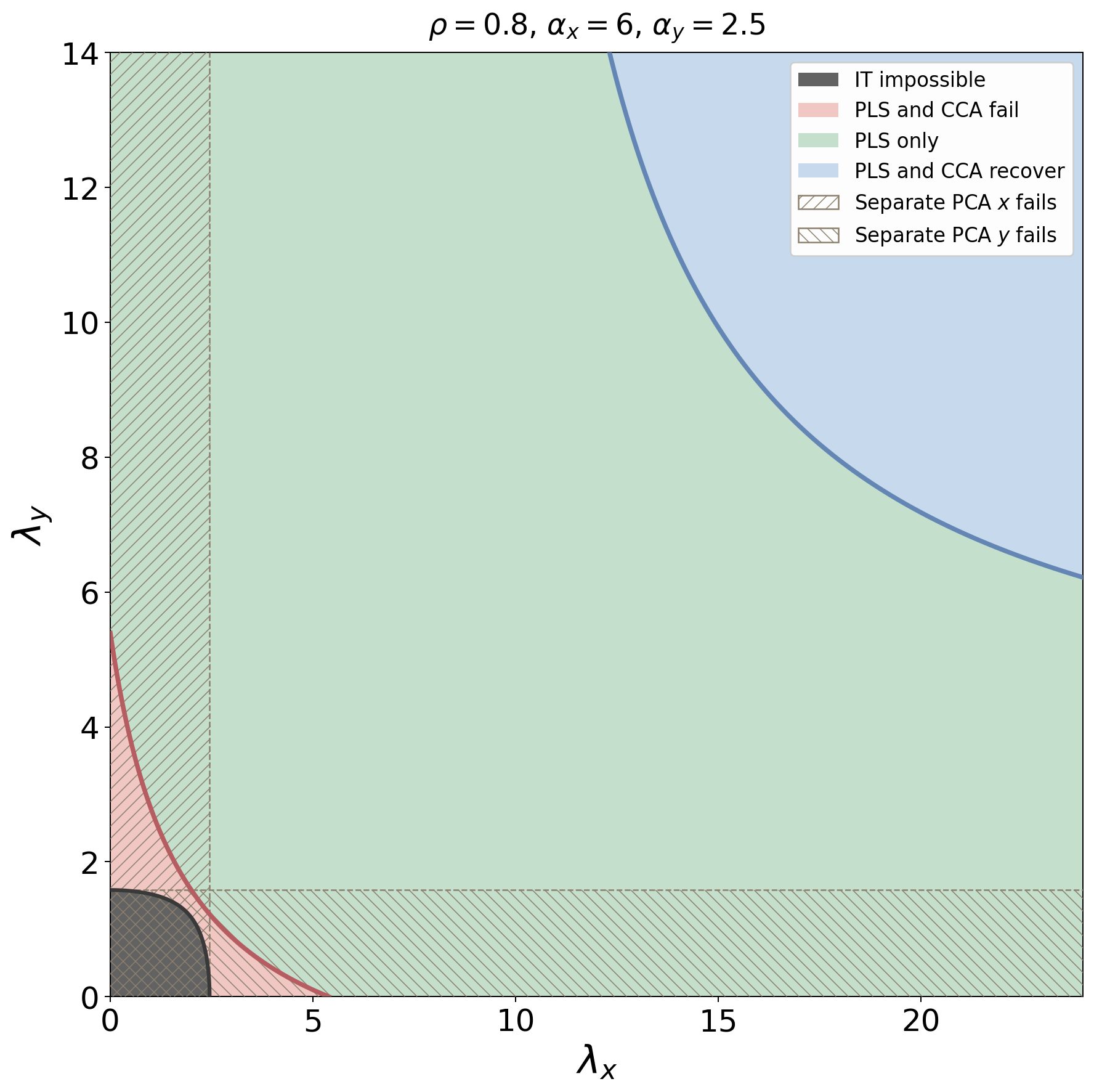}
    \end{subfigure}
    \caption{ Theoretical phase diagrams for detection in the $(\lambda_x,\lambda_y)$-plane for different values of the other parameters. The region of success for CCA is always included within the ones of PLS, as stated by Prop.~\ref{prop:PLS_vs_CCA}. }
    \label{add_fig:PLS_VS_CCA}
\end{figure*}

\section{Behavior of the Limiting Outliers and Overlaps with respect to the parameters}
\label{sec:formula_outlier_overlap}
In the rest of this section, we fix a component $k$ and omit its index and write $Q \equiv Q_{(\alpha_x,\alpha_y)}$ and $R \equiv R_{(\lambda_x, \lambda_y, \rho)}$ to ease notations. We assume $\lambda_x, \lambda_y>0$ and $|\rho| \in (0,1)$ with the limiting endpoints $\rho=0$ and $|\rho|=1$ treated below. We use the notation $(\cdot)_+ = \max(\cdot,0)$. 

From Theorem~\ref{thm:singularvalue}, the limiting singular values are given by
\begin{align}
\label{eq:sigmamp}
\sigma^\pm(\matSspk)&\to\sigma^\pm_\infty:=
\begin{cases}
\displaystyle\sqrt{\frac{(1+\mathrm r^\pm)(1+\alpha_x\mathrm r^\pm)(1+\alpha_y\mathrm r^\pm)}{\mathrm r^\pm}},
&\mbox{if }\mathrm r^\pm<\uptau^+,\\[1ex]
\varsigma_+,&\mbox{otherwise}.
\end{cases}
\end{align}
Similarly we recall from Theorem~\ref{thm:overlap} that the results for the overlaps write
\begin{align}
   \langle \rdmvect{\Tilde{u}}^\pm, \vect{v}^{\star}_{x} \rangle^2  
   &\to
   \mathrm{m}^\pm_x 
   = \begin{cases}
 \frac{( \mathrm{r}^\pm \lambda_y -1 ) Q_{(\alpha_x,\alpha_y)}(\mathrm{r}^\pm) }{\mathrm{r}^\pm(1+\mathrm{r}^\pm)(1 + \alpha_y \mathrm{r}^\pm) R_{(\lambda_x,\lambda_y,\rho)}^{\prime}(\mathrm{r}^\pm) }
 & 
 \mbox{if } 1/\mathrm{r}^\pm \geq 1/\uptau^+ \, ,\\\\
    0 
    & 
    \mbox{otherwise} \, ,
    \end{cases}
\end{align}
and
\begin{align}
 \langle \rdmvect{\Tilde{v}}^\pm, \vect{v}^{\star}_{y} \rangle^2  
   &\to
   \mathrm{m}^\pm_y  =
    \begin{cases}
 \frac{(\mathrm{r}^\pm \lambda_x -1 ) Q_{(\alpha_x,\alpha_y)}(\mathrm{r}^\pm) }{\mathrm{r}^\pm(1+\mathrm{r}^\pm)(1 + \alpha_x \mathrm{r}^\pm) R_{(\lambda_x,\lambda_y,\rho)}^{\prime}(\mathrm{r}^\pm) }
 & 
 \mbox{if } 1/\mathrm{r}^\pm \geq 1/\uptau^+ \, ,\\\\
    0 
    & 
    \mbox{otherwise} \, ,
    \end{cases}
\end{align}

\begin{prop}[Monotonicity in signal strength and correlation]
Assume that for $\sigma^\pm, \mathrm{m}_x^\pm, \mathrm{m}_y^\pm$ we are above the transition that is $1/\mathrm{r}^\pm > 1 /\uptau^+$, then we have 
\begin{itemize}
    \item \underline{for the outliers:}
     \begin{align}
         \frac{\partial \sigma^\pm }{\partial \lambda_x} > 0, \quad   \frac{\partial \sigma^\pm }{\partial \lambda_y} > 0
         \\
         \frac{\partial \sigma^- }{\partial \rho^2} > 0, \quad \frac{\partial \sigma^+ }{\partial \rho^2} < 0. 
     \end{align}
    \item \underline{for the squared overlaps:} 
    \begin{align}
         \frac{\partial \mathrm{m}^-_x }{\partial \lambda_x} > 0, \quad
\frac{\partial \mathrm{m}^-_y }{\partial \lambda_y} > 0, \\
\frac{\partial \mathrm{m}^+_x }{\partial \lambda_y} > 0, \quad
\frac{\partial \mathrm{m}^+_y }{\partial \lambda_x} > 0.
     \end{align}
     and furthermore if $\lambda_x \le \lambda_y$ then 
     \begin{align}
         \frac{\partial \mathrm{m}^-_x }{\partial \rho^2} > 0, \quad   \frac{\partial \mathrm{m}^+_x }{\partial \rho^2} < 0,
     \end{align}
     and the analogous statement holds for $\mathrm{m}^\pm_y$ if $\lambda_x \ge \lambda_y$. 
\end{itemize}
\end{prop}
\noindent Increasing either signal strength moves both outliers away from the fixed noise edge. Increasing $\rho^2$
instead moves
the leading outlier upwards and the secondary outlier downwards, changing how the signal alignment is shared
between the two singular vectors. Note also that increasing the signal of one channel (say $\lambda_x$) increases the overlap  $\mathrm{m}_x^-$ of the top singular vector of that channel with the corresponding signal but does not necessarily increases the overlap of the \emph{second} singular vector $\mathrm{m}_x^+$ with that same signal. An illustration of these two behaviors is given in Fig.~\ref{add_fig:evolution_lambda} for the signal strength and in  Fig.~\ref{add_fig:evolution_rho} for the correlation. 

\begin{proof}
    Differentiating the equation $R(\mathrm{r}^\pm)=0$ with respect to $\lambda_x, \lambda_y$ and $\rho^2$ gives thanks to the expression of Eq.~\eqref{eq:def_Pol_R}:
   \begin{align}
\frac{\partial\mathrm r^\pm}{\partial\lambda_x}
&=\frac{1-\lambda_y\mathrm r^\pm}{\lambda_x R'(\mathrm r^\pm)}, \quad \frac{\partial\mathrm r^\pm}{\partial\lambda_y}
=\frac{1-\lambda_x\mathrm r^\pm}{\lambda_y R'(\mathrm r^\pm)}, 
\qquad
\frac{\partial\mathrm r^\pm}{\partial\rho^2}
=\frac{\lambda_x\lambda_y\mathrm r^\pm(1+\mathrm r^\pm)}{R'(\mathrm r^\pm)}.
\end{align}
 Since
$\mathrm r^-<\min(1/\lambda_x,1/\lambda_y)$,
$\mathrm r^+>\max(1/\lambda_x,1/\lambda_y)$ and
$R'(\mathrm r^-)<0<R'(\mathrm r^+)$, both roots decrease with either signal strength, whereas $\mathrm r^-$ decreases and $\mathrm r^+$ increases with $\rho^2$. Furthermore,
\begin{align}
\frac{\mathrm d\,b(r)^2}{\mathrm dr}
&=-\frac{Q(r)}{r^2}<0,\qquad 0<r<\uptau^+,
\end{align}
which proves the statements for the outliers.
For the overlaps, write
\begin{align}
\mathrm m_x^\pm
&=
\bigg(-\frac{1-\lambda_y\mathrm r^\pm}{\mathrm r^\pm R'(\mathrm r^\pm)}\bigg)
\frac{Q(\mathrm r^\pm)}{(1+\mathrm r^\pm)(1+\alpha_y\mathrm r^\pm)}.
\end{align}
Both factors are positive above the transition. The second factor is strictly decreasing in $\mathrm r^\pm$, since $Q$ is positive and strictly decreasing on $(0,\uptau^+)$. Direct differentiation of the first factor gives
\begin{align}
\frac{\partial}{\partial\lambda_x}
\bigg(-\frac{1-\lambda_y\mathrm r^-}{\mathrm r^-R'(\mathrm r^-)}\bigg)
&=\frac{\rho^2\lambda_y(1+\lambda_y)(\lambda_x+\lambda_y+\rho^2\lambda_x\lambda_y)}{(-R'(\mathrm r^-))^3}>0,\\
\frac{\partial}{\partial\lambda_y}
\bigg(-\frac{1-\lambda_y\mathrm r^-}{\mathrm r^-R'(\mathrm r^-)}\bigg)
&=-\frac{\rho^2\lambda_x(\lambda_x+\lambda_y+(2-\rho^2)\lambda_x\lambda_y)}{(-R'(\mathrm r^-))^3}<0,\\
\frac{\partial}{\partial\rho^2}
\bigg(-\frac{1-\lambda_y\mathrm r^-}{\mathrm r^-R'(\mathrm r^-)}\bigg)
&=\frac{\lambda_x\lambda_y(1+\lambda_y)(\rho^2\lambda_x\lambda_y-\lambda_x+\lambda_y)}{(-R'(\mathrm r^-))^3}.
\end{align}
The last expression is positive if $\lambda_x\leq\lambda_y$. Combining these identities with the signs of the root derivatives proves the asserted inequalities for $\mathrm m_x^-$ and $\mathrm m_x^+$. Exchanging $x$ and $y$ proves the remaining statements. 
\end{proof}

\begin{figure*}[h]
    \centering
    \begin{subfigure}[t]{0.32\textwidth}
        \centering
\includegraphics[width=\textwidth]{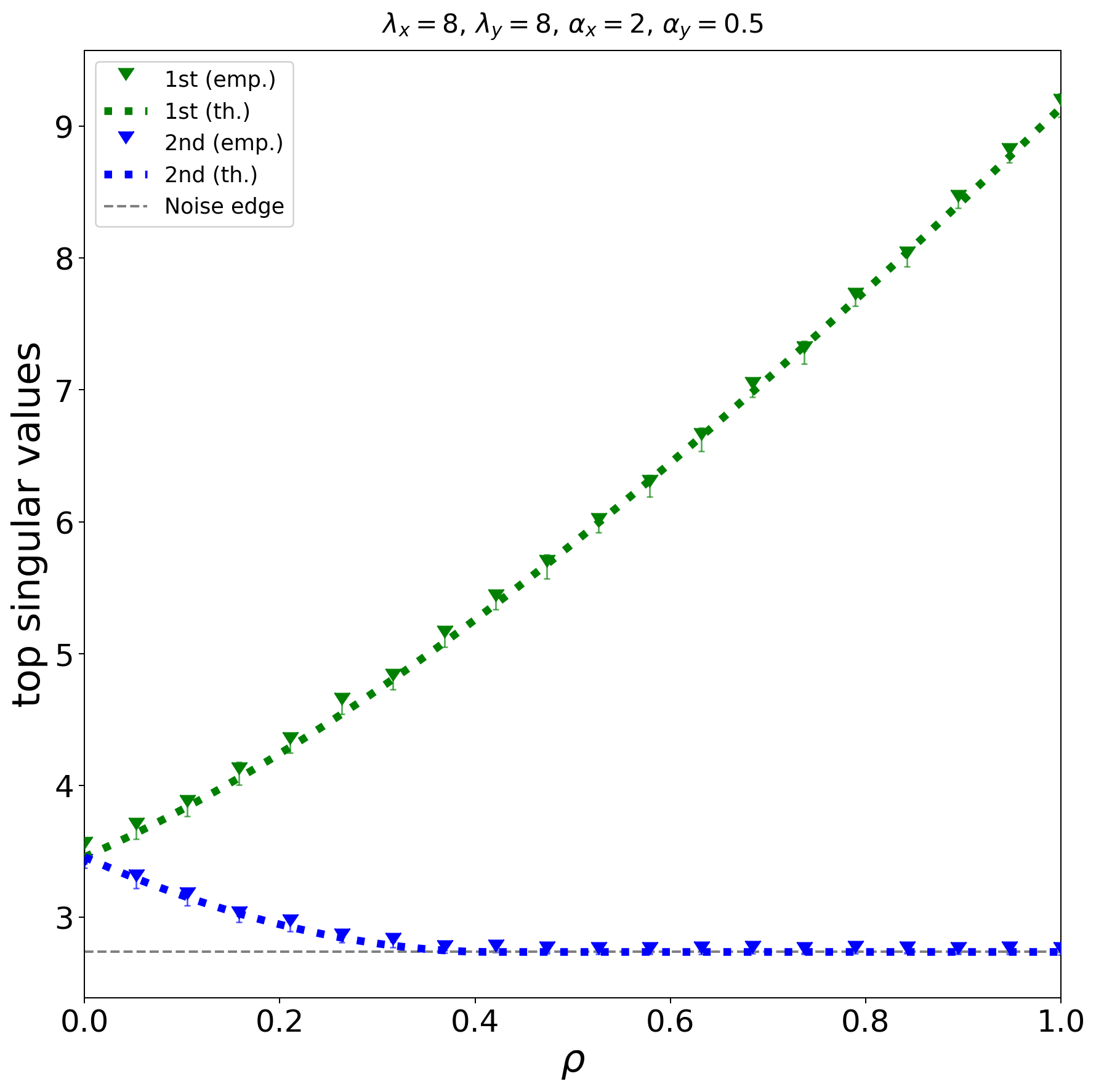}
    \end{subfigure}
    \hfill
    \begin{subfigure}[t]{0.32\textwidth}
        \centering
        \includegraphics[width=\textwidth]{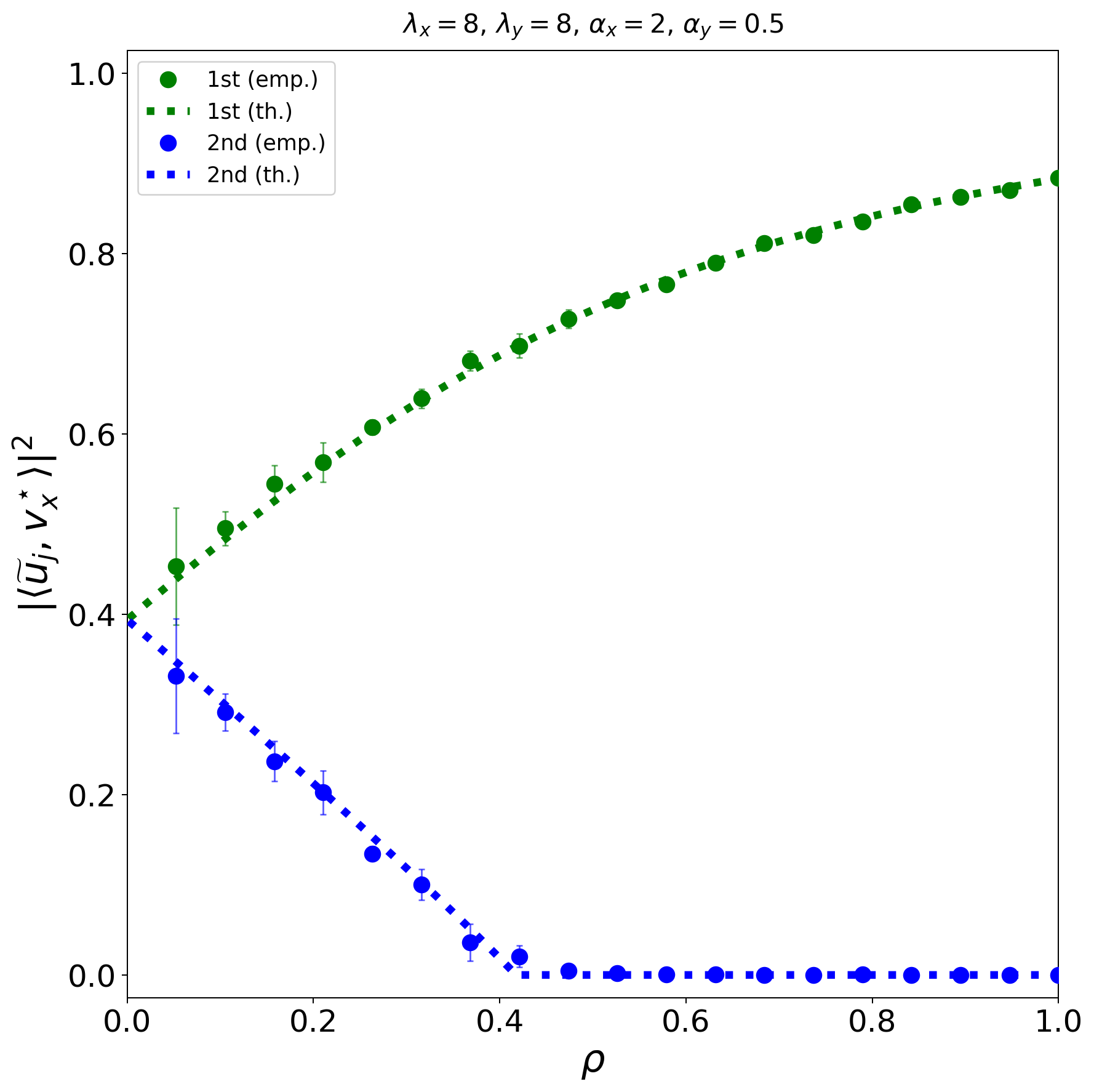}
    \end{subfigure}
    \hfill
    \begin{subfigure}[t]{0.32\textwidth}
        \centering
        \includegraphics[width=\textwidth]{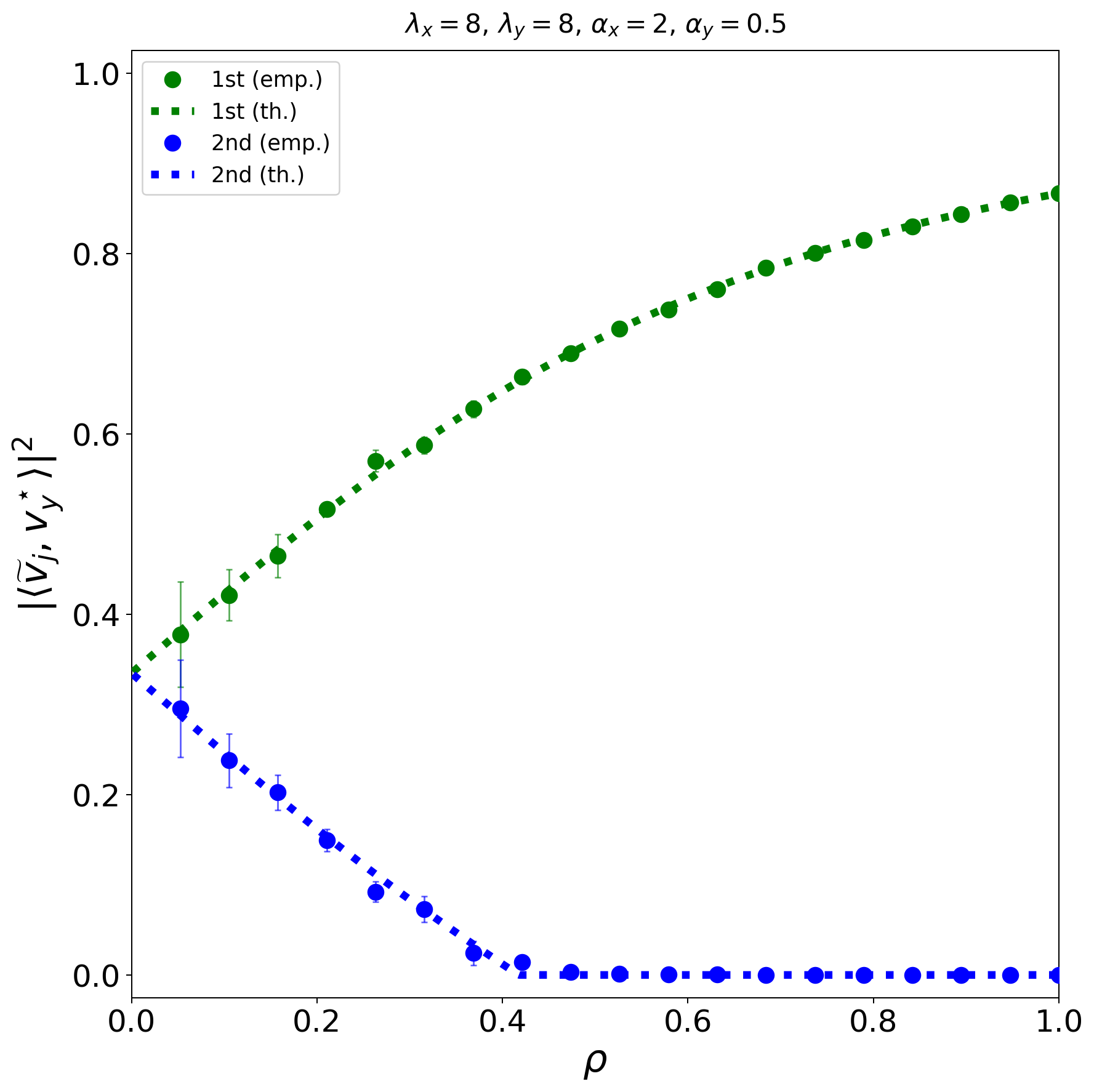}
    \end{subfigure}
    \caption{ Empirical and theoretical values of the top
two singular values (left) and the corresponding overlaps (center and right) of the spiked cross-covariance matrix with one spike $(r=1)$, plotted as functions
of the correlation $\rho$. All other parameters are fixed. Each empirical data point is
obtained as an average over $10$ samples. }
    \label{add_fig:evolution_rho}
\end{figure*}

\begin{figure*}[h]
    \centering
    \begin{subfigure}[t]{0.32\textwidth}
        \centering
\includegraphics[width=\textwidth]{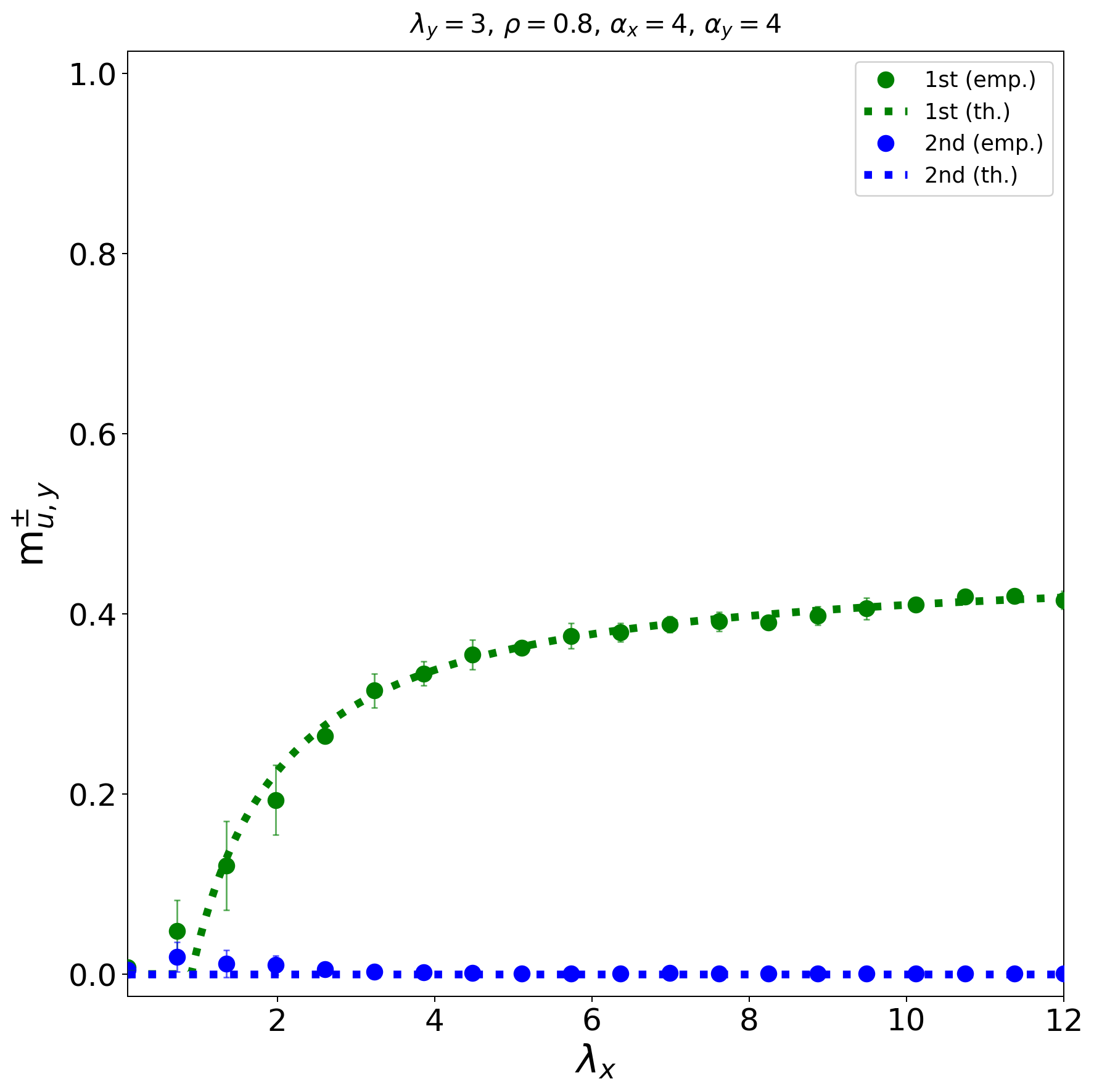}
    \end{subfigure}
    \hfill
    \begin{subfigure}[t]{0.32\textwidth}
        \centering
        \includegraphics[width=\textwidth]{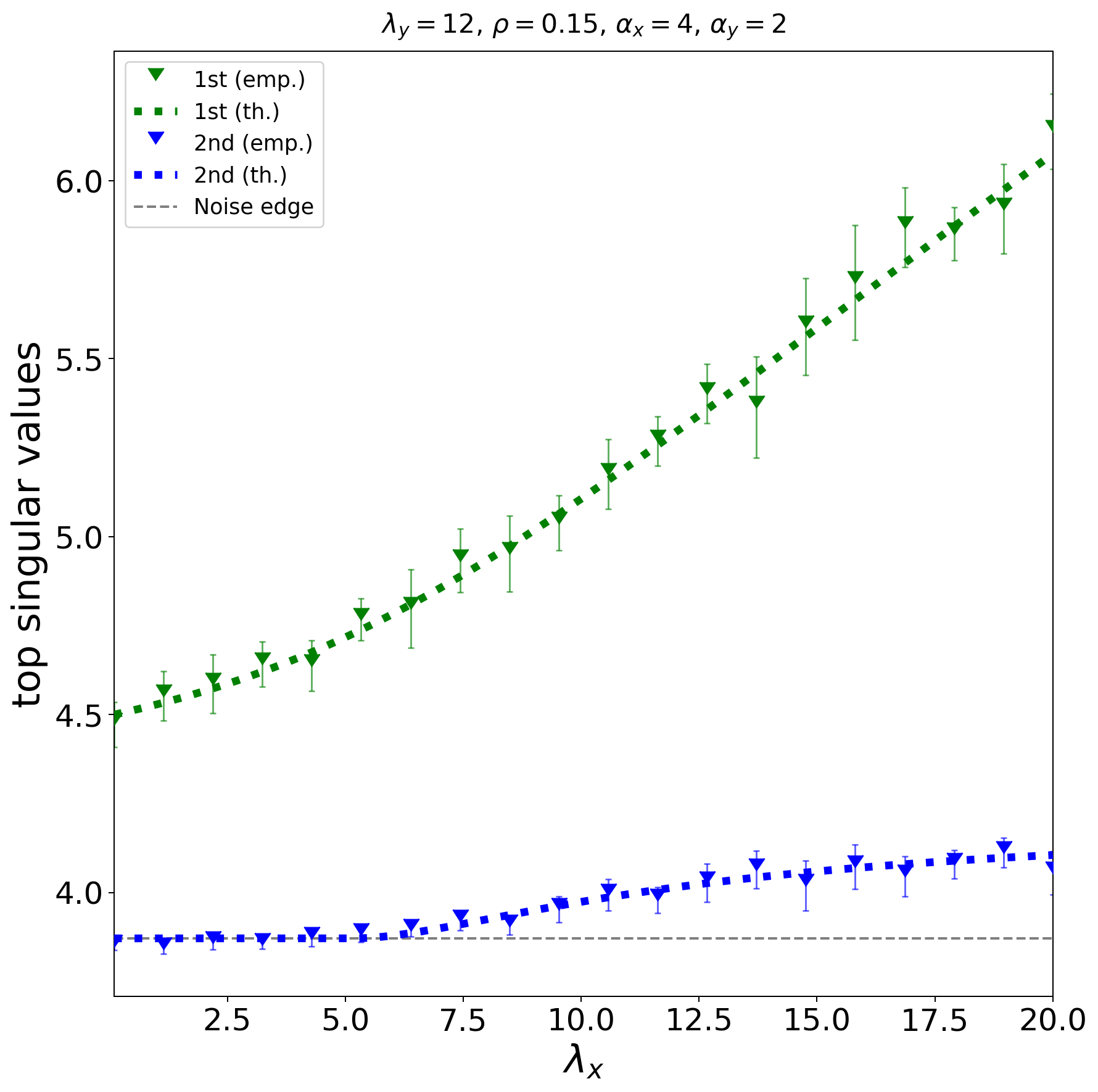}
    \end{subfigure}
    \hfill
    \begin{subfigure}[t]{0.32\textwidth}
        \centering
        \includegraphics[width=\textwidth]{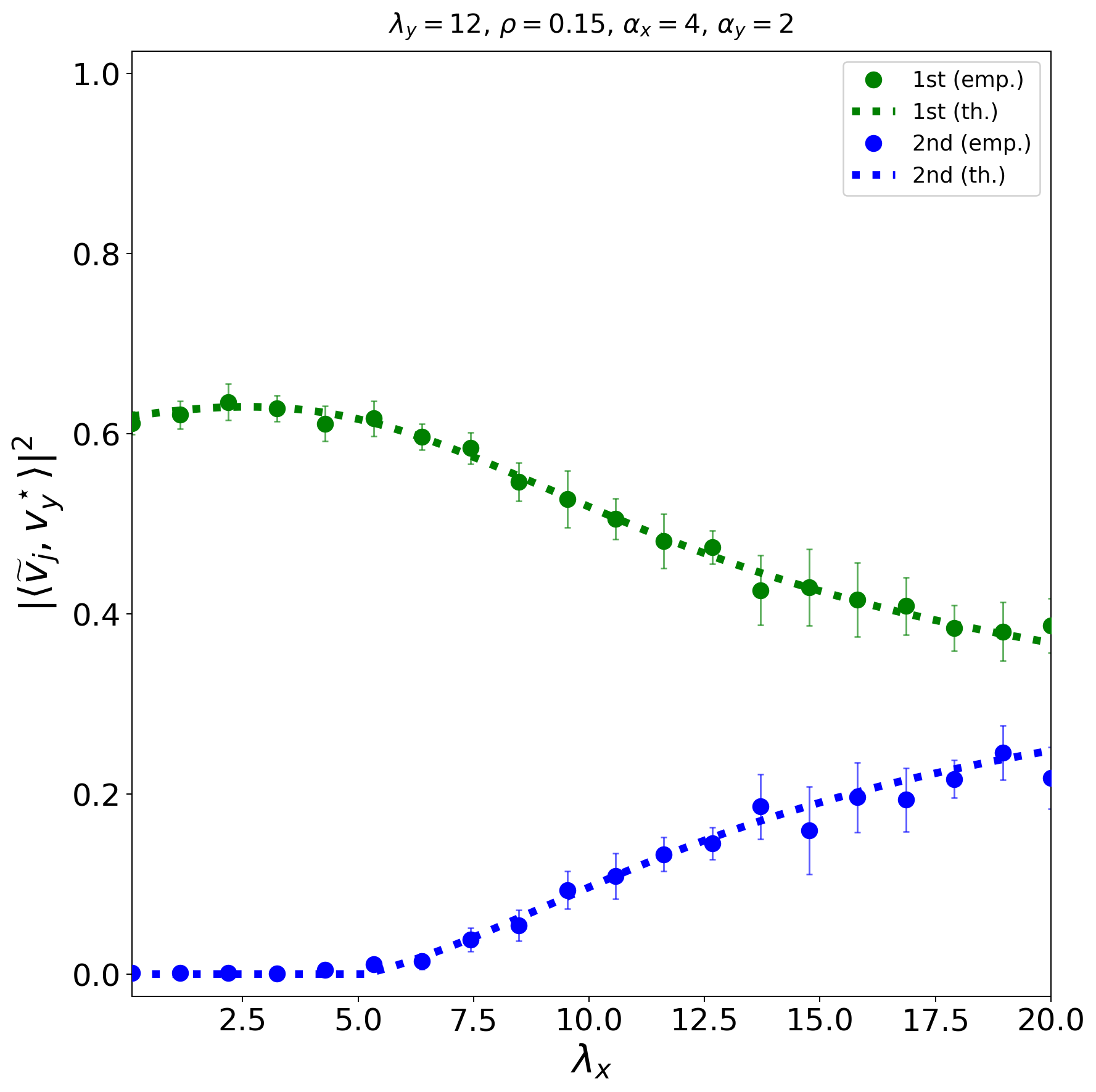}
    \end{subfigure}
    \caption{ Empirical and theoretical values of the top
two singular values (left) and the corresponding overlaps (center and right) of the spiked cross-covariance matrix with one spike $(r=1)$, plotted as functions
of the $x$-channel signal strength $\lambda_x$. All other parameters ($\lambda_y,\rho,\alpha_x,\alpha_y$) are fixed. Each empirical data point is
obtained as an average over $10$ samples. }
    \label{add_fig:evolution_lambda}
\end{figure*}

\paragraph{Simplification for perfectly aligned  signals ($|\rho|\to1$).} When the signals are perfectly aligned, we have from Eq.~\eqref{eq:behavior_root_rhoto1}, $\mathrm{r}^+ \to \infty$ hence the condition $1/\mathrm{r}^+ > 1/\uptau^+$ cannot be satisfied, thus $\sigma^+=\varsigma_+$ and there is at most one outlier which is given by 
\begin{align}
\mathrm r^-
&=
\frac{1}{\lambda_x+\lambda_y+\lambda_x\lambda_y}, && (|\rho|\to 1), \\
\sigma^-_\infty
&=
\begin{cases}
\sqrt{(1+\lambda_x)(1+\lambda_y)(1+\alpha_x\mathrm r^-)(1+\alpha_y\mathrm r^-)},
&\mbox{if }\mathrm r^-<\uptau^+,\\[1ex]
\varsigma_+,&\mbox{otherwise},
\end{cases}
&& (|\rho|\to 1).
\end{align}
and the corresponding overlaps are given by 
\begin{align}
\mathrm m_x^-
&=
\frac{\lambda_x}{1+\lambda_x}
\frac{\big(Q(\mathrm r^-)\big)_+}{1+\alpha_y\mathrm r^-},&
\mathrm m_y^-&=\frac{\lambda_y}{1+\lambda_y}
\frac{\big(Q(\mathrm r^-)\big)_+}{1+\alpha_x\mathrm r^-},\\
\mathrm m_x^+&=0,&\mathrm m_y^+&=0.
\end{align}

If both separate PCAs recover their signals, namely $\lambda_x>\sqrt{\alpha_x}$ and $\lambda_y>\sqrt{\alpha_y}$, then PLS also has an informative outlier. Indeed, $Q\big((\sqrt{\alpha_x}+\sqrt{\alpha_y}+\sqrt{\alpha_x\alpha_y})^{-1}\big)>0$, so $\mathrm r^-<\uptau^+$, that is we are in the regime where an outlier exists. Note that this does not imply that PLS succeeds whenever only one of the separate PCAs succeeds.

\paragraph{Simplification for uncorrelated signals ($\rho \to 0$).}For $\lambda_x>\lambda_y$, Eq.~\eqref{eq:behavior_root_rhoto0} gives $\mathrm r^-\to1/\lambda_x$ and $\mathrm r^+\to1/\lambda_y$, hence
\begin{align}
\sigma^-_\infty&= b(1/\lambda_x),&\sigma^+_\infty&= b(1/\lambda_y), && (|\rho|\to 0),
\\
\mathrm m_x^-&=\frac{\big(Q_{(\alpha_x,\alpha_y)}(1/\lambda_x)\big)_+}{(1+1/\lambda_x)(1+\alpha_y/\lambda_x)},&\mathrm m_y^-&=0, && (|\rho|\to 0),
\\
\mathrm m_y^+&=\frac{\big(Q_{(\alpha_x,\alpha_y)}(1/\lambda_y)\big)_+}{(1+1/\lambda_y)(1+\alpha_x/\lambda_y)},&\mathrm m_x^+&=0, && (|\rho|\to 0).
\end{align}
For $\lambda_y>\lambda_x$, exchange $x$ and $y$ in these expressions.

\paragraph{Simplification for equal size datasets ($\alpha_x=\alpha_y=\alpha$).} In this case, the edge polynomial factors as
\begin{align}
Q(r)&=(1+\alpha r)(1-\alpha r-2\alpha r^2),&
\uptau^+&=\frac{2}{\alpha+\sqrt{\alpha^2+8\alpha}}.
\end{align}
Consequently, the limiting singular values and overlaps simplify to
\begin{align}
\sigma^\pm_\infty&=
\begin{cases}
(1+\alpha\mathrm r^\pm)\sqrt{(1+\mathrm r^\pm)/\mathrm r^\pm},
&\mbox{if }\mathrm r^\pm<\uptau^+,\\[1ex]
\varsigma_+,&\mbox{otherwise}, 
\end{cases} && (\alpha_x = \alpha_y = \alpha),
\\
\mathrm m_x^\pm&=-\frac{(1-\lambda_y\mathrm r^\pm)
\big(1-\alpha\mathrm r^\pm-2\alpha(\mathrm r^\pm)^2\big)_+}{\mathrm r^\pm(1+\mathrm r^\pm)R'(\mathrm r^\pm)}, && (\alpha_x = \alpha_y = \alpha),
\\
\mathrm m_y^\pm&=-\frac{(1-\lambda_x\mathrm r^\pm)
\big(1-\alpha\mathrm r^\pm-2\alpha(\mathrm r^\pm)^2\big)_+}{\mathrm r^\pm(1+\mathrm r^\pm)R'(\mathrm r^\pm)}, && (\alpha_x = \alpha_y = \alpha).
\end{align}

\section{Score Overlaps}
\label{sec:app:score_overlaps}
We now give the overlaps of the scores obtained at step~(2) of the PLS algorithm. We fix a component $k$ and omit its index as in the previous section. The notation $\mathrm m_x^\pm,\mathrm m_y^\pm$ still denotes the overlaps of the singular vectors (see App.~\ref{sec:formula_outlier_overlap}).
\begin{prop}[Overlaps for the Score]
\label{prop:score_overlaps}
Assume $\lambda_x,\lambda_y>0$ and let $\mathrm r^\pm<\uptau^+$ correspond to a simple separated outlier $\sigma^\pm=b(\mathrm r^\pm)$. Under Assumptions~(A1)-(A2), the normalized scores satisfy
\begin{align}
\bigg\langle\usx,
\frac{\matXspk\rdmvect{\Tilde u}^{\pm}}
{\|\matXspk\rdmvect{\Tilde u}^{\pm}\|}\bigg\rangle^2
&\xrightarrow[n\to\infty]{\as}\mathrm m_{u,x}^{\pm} :=\frac{(1+\lambda_x)^2\mathrm r^\pm(1+\alpha_y\mathrm r^\pm)}
{\lambda_x(1+\mathrm r^\pm)(1+\alpha_x\mathrm r^\pm)}
\frac{\mathrm m_x^\pm}{1-\mathrm m_y^\pm} ,
\\
\bigg\langle\usy,
\frac{\matYspk\rdmvect{\Tilde v}^{\pm}}
{\|\matYspk\rdmvect{\Tilde v}^{\pm}\|}\bigg\rangle^2
&\xrightarrow[n\to\infty]{\as}\mathrm m_{u,y}^{\pm} := \frac{(1+\lambda_y)^2\mathrm r^\pm(1+\alpha_x\mathrm r^\pm)}
{\lambda_y(1+\mathrm r^\pm)(1+\alpha_y\mathrm r^\pm)}
\frac{\mathrm m_y^\pm}{1-\mathrm m_x^\pm}.
\label{eq:score_overlaps_explicit}
\end{align}
\end{prop}

\begin{proof}
We fix a simple separated outlier and omit its component and branch indices. Thus $z=\sigma^\pm$, $r=\mathrm r^\pm$ and $\rdmvect{\Tilde u},\rdmvect{\Tilde v}$ are its unit-norm singular vectors. We also write $\mathrm m_x=\mathrm m_x^\pm$ and $\mathrm m_y=\mathrm m_y^\pm$. We first give the calculation for one component.

As in Lemma~\ref{lem:overlap_resolvent}, the diagonal entries $(22)$ and $(44)$ of the block resolvent give
\begin{align}
\langle\usy,\matX\rdmvect{\Tilde u}\rangle^2
&\to2\frac{f_{2,k}(z)}{j_k'(z)},&
\langle\usx,\matY\rdmvect{\Tilde v}\rangle^2
&\to2\frac{f_{4,k}(z)}{j_k'(z)}.
\label{eq:scores:blockresidues}
\end{align}
To compute the normalization, we use the low-rank decomposition of the data matrix:
\begin{align}
\|\matXspk\rdmvect{\Tilde u}\|^2
&=\|\matX\rdmvect{\Tilde u}\|^2
+2\sqrt{\lambda_x}\langle\usx,\matX\rdmvect{\Tilde u}\rangle
\langle\vsx,\rdmvect{\Tilde u}\rangle
+\lambda_x\langle\vsx,\rdmvect{\Tilde u}\rangle^2.
\label{eq:scores:norm_expansion}
\end{align}
Note however that the singular vectors depend on the noise, so Lemma~\ref{lem:concentrationquadraticform} cannot be applied directly. We first express these quadratic forms through the noise resolvent, starting with (i) the second and third terms of the sum in Eq.~\eqref{eq:scores:norm_expansion}  containing the signals $\vsx,\usx$ and then (ii) dealing with the first noisy term.  

\paragraph{(i) Separating the signal directions.}
Let $\mat U$ have orthonormal columns spanning $\mathrm{Span}(\usx,\usy)$, and write $\vect e_x=\mat U^T\usx$, $\vect e_y=\mat U^T\usy$. These are unit vectors with $\langle\vect e_x,\vect e_y\rangle\to\rho$. Define
\begin{align}
\mat X_0&=(\bm I-\mat U\mat U^T)\matX,&
\mat Y_0&=(\bm I-\mat U\mat U^T)\matY,&
\mat S_0&=\mat X_0^T\mat Y_0,\\
\mat B_x&=\matX^T\mat U+\sqrt{\lambda_x}\vsx\vect e_x^T,&
\mat B_y&=\matY^T\mat U+\sqrt{\lambda_y}\vsy\vect e_y^T.
\end{align}
Then, exactly,
\begin{align}
\matXspk&=\mat X_0+\mat U\mat B_x^T,&
\matYspk&=\mat Y_0+\mat U\mat B_y^T,&
\matSspk&=\mat S_0+\mat B_x\mat B_y^T.
\end{align}
Conditional on the signals, $\mat B_x,\mat B_y$ are independent of $\mat X_0,\mat Y_0$. Removing finitely many sample directions leaves the limiting normalized traces used to compute $h_1,\ldots,h_4$ unchanged. Put
\begin{align}
\vect a&=\mat B_x^T\rdmvect{\Tilde u},&
\vect b&=\mat B_y^T\rdmvect{\Tilde v},&
\mat\Gamma_x(z)&=(z^2\bm I-\mat S_0\mat S_0^T)^{-1},&
\mat\Gamma_y(z)&=(z^2\bm I-\mat S_0^T\mat S_0)^{-1}.
\end{align}
The singular-vector equations and the block decomposition of the chiral resolvent give
\begin{align}
\rdmvect{\Tilde u}
&=z\mat\Gamma_x\mat B_x\vect b
+\mat\Gamma_x\mat S_0\mat B_y\vect a,\\
\rdmvect{\Tilde v}
&=\mat\Gamma_y\mat S_0^T\mat B_x\vect b
+z\mat\Gamma_y\mat B_y\vect a.
\label{eq:scores:noise_representation}
\end{align}

Expanding $\mat B_x,\mat B_y$ and applying Lemma~\ref{lem:concentrationquadraticform} to the independent signal directions gives
\begin{align}
\vect a&=h_1(\bm I+\lambda_x\vect e_x\vect e_x^T)\vect b+o(1),&
\vect b&=h_3(\bm I+\lambda_y\vect e_y\vect e_y^T)\vect a+o(1),\\
\langle\vsx,\rdmvect{\Tilde u}\rangle
&=\sqrt{\lambda_x}h_1\vect e_x^T\vect b+o(1),&
\langle\vsy,\rdmvect{\Tilde v}\rangle
&=\sqrt{\lambda_y}h_3\vect e_y^T\vect a+o(1).
\label{eq:scores:coordinate_limits}
\end{align}

Since $\langle\usx,\matXspk\rdmvect{\Tilde u}\rangle=\vect e_x^T\vect a$, Eq.~\eqref{eq:scores:coordinate_limits} gives
\begin{align}
\langle\usx,\matXspk\rdmvect{\Tilde u}\rangle
&=\frac{1+\lambda_x}{\sqrt{\lambda_x}}
\langle\vsx,\rdmvect{\Tilde u}\rangle+o(1),\\
\langle\usx,\matX\rdmvect{\Tilde u}\rangle
&=\frac1{\sqrt{\lambda_x}}
\langle\vsx,\rdmvect{\Tilde u}\rangle+o(1).
\label{eq:scores:own_projection}
\end{align}
Thus the second and third terms of Eq.~\eqref{eq:scores:norm_expansion} converge to $2\mathrm m_x$ and $\lambda_x\mathrm m_x$, respectively.

\paragraph{(ii) The noise quadratic form.}
We first compute $\|\mat X_0\rdmvect{\Tilde u}\|^2$ using Eq.~\eqref{eq:scores:noise_representation}. The same concentration argument makes the mixed term vanish and gives
\begin{align}
\|\mat X_0\rdmvect{\Tilde u}\|^2
&=\frac{z^2}{d_x}\mathrm{Tr}(\mat X_0^T\mat X_0\mat\Gamma_x^2)
\vect b^T(\bm I+\lambda_x\vect e_x\vect e_x^T)\vect b\notag\\
&\quad+\frac1{d_y}\mathrm{Tr}(\mat S_0^T\mat\Gamma_x\mat X_0^T\mat X_0\mat\Gamma_x\mat S_0)
\vect a^T(\bm I+\lambda_y\vect e_y\vect e_y^T)\vect a+o(1).
\end{align}
The trace limit already used to compute $h_2$ is
\begin{align}
\frac1n\mathrm{Tr}(\mat X_0\mat\Gamma_x(z)\mat X_0^T)
&\to\frac{h_2(z)}z.
\end{align}
Differentiating this locally uniform analytic limit and using $\mat\Gamma_x\mat S_0\mat S_0^T\mat\Gamma_x=z^2\mat\Gamma_x^2-\mat\Gamma_x$ yields
\begin{align}
\frac1{d_x}\mathrm{Tr}(\mat X_0^T\mat X_0\mat\Gamma_x^2)
&\to-\frac{\alpha_x}{2z}\bigg(\frac{h_2(z)}z\bigg)',\\
\frac1{d_y}\mathrm{Tr}(\mat S_0^T\mat\Gamma_x\mat X_0^T\mat X_0\mat\Gamma_x\mat S_0)
&\to\alpha_y\bigg[-\frac z2\bigg(\frac{h_2(z)}z\bigg)'-\frac{h_2(z)}z\bigg].
\end{align}
Consequently, Eq.~\eqref{eq:scores:coordinate_limits} reduces the quadratic form to
\begin{align}
\|\mat X_0\rdmvect{\Tilde u}\|^2
&=\langle\vect a,\vect b\rangle
\bigg[-\frac{\alpha_xz}{2h_1}\bigg(\frac{h_2}z\bigg)'
+\frac{\alpha_y}{h_3}\bigg\{-\frac z2\bigg(\frac{h_2}z\bigg)'-\frac{h_2}z\bigg\}\bigg]+o(1).
\label{eq:scores:bulk_norm_trace}
\end{align}
To determine $\langle\vect a,\vect b\rangle$, apply the same calculation to $\|\rdmvect{\Tilde u}\|^2+\|\rdmvect{\Tilde v}\|^2=2$. Since $\rdmmat G_0'(z)=-\rdmmat G_0(z)^2$ for $\rdmmat G_0(z)=(z\bm I-\girko(\mat S_0))^{-1}$, this gives
\begin{align}
2
&=-h_1'\vect b^T(\bm I+\lambda_x\vect e_x\vect e_x^T)\vect b
-h_3'\vect a^T(\bm I+\lambda_y\vect e_y\vect e_y^T)\vect a+o(1)\notag\\
&=-\langle\vect a,\vect b\rangle
\bigg(\frac{h_1'}{h_1}+\frac{h_3'}{h_3}\bigg)+o(1).
\end{align}
At $z=\sigma^\pm$ and $r=\mathrm r^\pm$, the previously obtained formulas give
\begin{align}
h_1&=\frac{1+\alpha_xr}z,&
h_3&=\frac{1+\alpha_yr}z,&
\frac{h_2}z&=\frac r{1+\alpha_yr},\\
\bigg(\frac{h_2}z\bigg)'
&=-\frac{2zr^2}{Q(r)(1+\alpha_yr)^2},&
\frac{h_1'}{h_1}+\frac{h_3'}{h_3}
&=-\frac{2zr}{Q(r)(1+r)}.
\end{align}
Hence
\begin{align}
\langle\vect a,\vect b\rangle
&\to\frac{Q(r)(1+r)}{zr},\\
\|\mat X_0\rdmvect{\Tilde u}\|^2
&\to\frac{\alpha_x(1+r)^2}{1+\alpha_yr}
+\frac{\alpha_y(1+r)\big((1+r)(1+\alpha_xr)-Q(r)\big)}{(1+\alpha_yr)^2}\notag\\
&=\frac{(1+r)(1+\alpha_xr-Q(r))}{r(1+\alpha_yr)}
=\frac1{h_3^2}-\frac{Q(r)(1+r)}{zr h_3}.
\label{eq:scores:bulk_norm_limit}
\end{align}
The simplification uses $z^2=(1+r)(1+\alpha_xr)(1+\alpha_yr)/r$ and
\begin{align}
Q(r)+r(1+r)(\alpha_x+\alpha_y+2\alpha_x\alpha_yr)
&=(1+\alpha_xr)(1+\alpha_yr).
\end{align}
On the other hand, Eq.~\eqref{eq:scores:coordinate_limits} implies
\begin{align}
\|\vect a\|^2
&=\frac{\langle\vect a,\vect b\rangle}{h_3}
-\lambda_y(\vect e_y^T\vect a)^2+o(1)
\to\frac{Q(r)(1+r)}{zr h_3}-\frac{\mathrm m_y}{h_3^2}.
\end{align}
Since $\mat X=\mat X_0+\mat U(\mat B_x^T-\sqrt{\lambda_x}\vect e_x(\vsx)^T)$ is an orthogonal decomposition in sample space, we obtain
\begin{align}
\|\matX\rdmvect{\Tilde u}\|^2
&=\|\mat X_0\rdmvect{\Tilde u}\|^2
+\big\|\vect a-\sqrt{\lambda_x}\vect e_x\langle\vsx,\rdmvect{\Tilde u}\rangle\big\|^2\notag\\
&\to\frac{1-\mathrm m_y}{h_3^2}-(\lambda_x+2)\mathrm m_x.
\end{align}
Adding the other two terms of Eq.~\eqref{eq:scores:norm_expansion} therefore gives
\begin{align}
\|\matXspk\rdmvect{\Tilde u}\|^2
&\to\frac{1-\mathrm m_y}{h_3^2}
=\frac{z^2}{(1+\alpha_yr)^2}(1-\mathrm m_y),\\
\|\matYspk\rdmvect{\Tilde v}\|^2
&\to\frac{1-\mathrm m_x}{h_1^2}
=\frac{z^2}{(1+\alpha_xr)^2}(1-\mathrm m_x).
\label{eq:scores:schurbridge}
\end{align}
The second identity follows by exchanging $x,y$. Squaring Eq.~\eqref{eq:scores:own_projection} and dividing by Eq.~\eqref{eq:scores:schurbridge} now gives
\begin{align}
\mathrm m_{u,x}^{\pm}
&=\frac{(1+\lambda_x)^2(1+\alpha_y\mathrm r^\pm)^2}
{\lambda_x(\sigma^\pm)^2}\frac{\mathrm m_x^\pm}{1-\mathrm m_y^\pm},
\end{align}
which is the stated expression after substituting the value of $(\sigma^\pm)^2$. The $y$ expression follows identically. 
\end{proof}

\section{Optimal Angle of Rotations}
\label{sec:app:angle_rotation}
Since both left singular vector $\rdmvect{\Tilde{u}}_+$ and $\rdmvect{\Tilde{u}}_-$ may correlate with the signal components $\vsxk$, one may as well construct a class of unit-norm estimator by performing a rotation in the plane of this two orthogonal vectors, the latter can be parametrized by $\beta \in [0,1]$ in the following way, by  choosing the signs of the two singular vectors
so that their overlaps with the signal vectors are nonnegative.
\begin{align}
    \rdmvect{\hat{w}}_x(\beta) := \beta \rdmvect{\Tilde{u}}_- + \sqrt{1 - \beta^2}  \rdmvect{\Tilde{u}}_+ \, ,
\end{align}
and whose (squared) overlap with $\vsxk$ is given asymptotically by 
\begin{align}
    \langle \vsxk, \rdmvect{\hat{w}}_x(\beta) \rangle^2 \xrightarrow[n \to \infty]{\as} \mathrm{q}_k(\beta) =\bigg( \beta  \, \sqrt{\mathrm{m}^-_{x,k}} + \sqrt{1 -\beta^2}  \, \sqrt{\mathrm{m}^+_{x,k}} \bigg)^2 \, .
\end{align}
and so to optimize with respect to $\beta$ one simplify solves $\mathrm{q}_k'(\beta) = 0$ from which one finds that the optimal angle of rotation is given by 
\begin{align}
    \beta^{\mathrm{opt}} = \sqrt{\frac{\mathrm{m}^-_{x,k} }{ \mathrm{m}^-_{x,k} + \mathrm{m}^+_{x,k}}}
    \, . 
\end{align}

\end{document}